\documentclass{amsart}

\usepackage{amsmath,amssymb,amsthm,mathtools,mathrsfs}
\usepackage{enumerate}
\usepackage{booktabs}
\usepackage{array}
\usepackage{amssymb}
\usepackage{enumerate}
\usepackage{color}
\usepackage{esint}

\usepackage[
colorlinks=true,      
linkcolor=blue,      
citecolor=blue,       
urlcolor=blue       
]{hyperref}

\newcommand{\N}{\mathbb N}

\numberwithin{equation}{section}

\newtheorem{thm}{Theorem}[section]
\newtheorem{lem}[thm]{Lemma}
\newtheorem{prop}[thm]{Proposition}

\theoremstyle{definition}

\theoremstyle{remark}
\newtheorem{rem}[thm]{Remark}

\newcommand{\1}{\mathbf 1}
\newcommand{\BMO}{\mathrm{BMO}}
\newcommand{\T}{\mathsf T}

\title{Sharp Fractional Riesz Estimates on the Hypercube}

\author[Jiao]{Yong Jiao}
\address{School of Mathematics and Statistics, HNP-LAMA, Central South University, Changsha 410075, China}
\email{jiaoyong@csu.edu.cn}

\author[Luo]{Sijie Luo}
\address{School of Mathematics and Statistics, HNP-LAMA, Central South University, Changsha 410075, China}
\email{sijieluo@csu.edu.cn}

\author[Zanin]{Dmitriy Zanin}
\address{School of Mathematics and Statistics, HNP-LAMA, Central South University, Changsha 410075, China}
\email{d.zanin@csu.edu.cn}

\author[Zhou]{Dejian Zhou}
\address{School of Mathematics and Statistics, HNP-LAMA, Central South University, Changsha 410075, China}
\email{zhoudejian@csu.edu.cn}


\date{}

\begin{document}

\begin{abstract}
Let $\Omega_{n}=\{-1,1\}^n$ be the $n$-dimensional hypercube equipped with the normalized uniform measure, let $\nabla$ be the Walsh gradient and let $\Delta$ be the Walsh Laplacian.  For
every $1<p\leq 2$ we prove the following estimate
\[
\|\nabla f\|_{L_p(\Omega_n;\ell_2^n)} \leq c_{{\rm abs}}(p-1)^{-2}\|\Delta^{1/p}f\|_{L_p(\Omega_n)}.
\]
The exponent $\frac1p$ is optimal, thus this settles the open problem on the sharp fractional Riesz estimate by Efraim and Lust-Piquard \cite{E-LP2008} which was subsequently highlighted by Ivanisvili and Volberg \cite{I-V2022}. We also establish the higher-order counterpart. As applications of our results, we obtain simpler proofs of the optimal short-time estimate for $\nabla e^{-t\Delta}$, and the Bernstein-Markov type inequality for $d$-bounded degree functions.
\end{abstract}

\maketitle


\section{Introduction}

For $n\in\N$, let $\Omega_{n}=\{-1,1\}^{n}$ be the $n$-dimensional hypercube with normalized uniform measure and let $L_{p}(\Omega_{n})$ be the usual $L_{p}$ space for $1\leq p\leq \infty$.  We write $D_{j}$ for the $j$-th Walsh derivative, $\nabla f=\{D_{j}f\}_{j=1}^{n}$ for the discrete gradient, $\Delta=\sum_{j=1}^{n}D_{j}$ for the Walsh number operator, and $P_{t}=e^{-t\Delta}$ for the heat semigroup. For more details on the analysis of hypercubes such as Fourier analysis theory and spectral calculus, we refer to Section \ref{sec:preliminaries}.

The analysis on hypercube lies at interface between Fourier analysis, martingale theory, and Markov semigroups, which play an essential role in computer science and information theory. For the Hilbertian case, the following identity holds
\begin{equation}\label{eq:intro-energy}
 \|\nabla f\|_{L_{2}(\Omega_{n};\ell_{2}^{n})}^{2}
 =\sum_{j=1}^{n}\|D_{j}f\|_{L_{2}(\Omega_{n})}^{2}
 =\langle\Delta f,f\rangle
 =\|\Delta^{1/2}f\|_{L_{2}(\Omega_{n})}^{2}.
\end{equation}
Accordingly, $\nabla\Delta^{-1/2}$ is an isometry from $L_{2}(\Omega_{n})$ into $L_{2}(\Omega_{n};\ell_{2}^{n})$.  The classical estimate for the (discrete) Riesz transform focus on this comparison on $L_{p}(\Omega_{n})$ with constants independent of $n$. In 1998, Lust-Piquard \cite{LustPiquard1998} used tools from noncommutative analysis to prove that for every $2<p<\infty$, 
\begin{equation}\label{eq:classical-upper-riesz} 
\|\nabla f\|_{L_{p}(\Omega_{n};\ell_{2}^{n})} \asymp_{p} \|\Delta^{1/2}f\|_{L_{p}(\Omega_{n})}, 
\end{equation} 
with dimension-free constants. However, the situation is fundamentally different when $1<p<2$. In the same paper, Lust-Piquard recorded a counterexample communicated by Lamberton showing that the upper estimate
\begin{equation}\label{eq:failed-half-power-riesz}
\|\nabla f\|_{L_{p}(\Omega_{n};\ell_{2}^{n})} \leq C_{p}\|\Delta^{1/2}f\|_{L_{p}(\Omega_{n})}
\end{equation}
cannot hold with a constant $C_{p}$ independent of $n$. This obstruction is commonly referred to as Lamberton's counterexample.

\noindent For $0<\beta<1$ and $n\in\mathbb N$, let
\[
\mathfrak R_{p,\beta}(n)=\left\|\nabla\Delta^{-\beta}\right\|_{L_{p}(\Omega_{n})\to L_{p}(\Omega_{n};\ell_{2}^{n})}
\]
be  the smallest constant such that
\[
 \|\nabla f\|_{L_{p}(\Omega_{n};\ell_{2}^{n})} \leq\mathfrak R_{p,\beta}(n)\|\Delta^{\beta}f\|_{L_{p}(\Omega_{n})}
\]
holds for every $f:\Omega_{n}\to\mathbb C$. Since half-power estimate fails for $1<p<2$, this leads to the following fundamental question: for which exponents $0<\beta<1$ does one have 
\[
\sup_{n\geq1}\mathfrak R_{p,\beta}(n)<\infty?
\]
Any such exponent must satisfy $\beta\geq1/p$ whenever $1<p<2$. This necessary condition follows from an argument attributed to Naor and Schechtman. This was stated by Efraim and Lust-Piquard in \cite[Lemma 5.5]{E-LP2008}. Indeed, for $F_n=\1_{\{(1,\ldots,1)\}},$ we have
\[
\|\nabla F_n\|_{L_p(\Omega_n;\ell_2^n)}\gtrsim_p n^{\frac1p}\|F_n\|_{L_p(\Omega_n)},\quad \|\Delta^{\beta} F_n\|_{L_p(\Omega_n)}\lesssim n^{\beta}\|F_n\|_{L_p(\Omega_n)},\quad 1<p<2.
\]
Hence, $\sup_{n\geq1}\mathfrak{R}_{p,\beta}(n)<\infty$ only if $\beta\geq\frac{1}{p}$. We now summary known results on the dimension-dependent upper bounds of $\mathfrak R_{p,\beta}(n)$ as follows.
\begin{prop}[\cite{E-LP2008, IVHV2020}]\label{prop:intro-known-dimensional}
Let $1<p<2$ and $0<\beta<1$.  Then
\begin{equation}\label{eq:intro-known-dimensional-upper}
 \mathfrak R_{p,\beta}(n)
 \lesssim_{p,\beta}
 \begin{cases}
  n^{1/p-\beta},
     &0<\beta<1/p,\\[2mm]
  1+\log n,
     &\beta=1/p,\\[2mm]
  1,
     &1/p<\beta<1.
 \end{cases}
\end{equation}
Moreover, when $0<\beta<1/p$,
\begin{equation}\label{eq:intro-known-dimensional-lower}
 \mathfrak R_{p,\beta}(n)
 \gtrsim_{p,\beta}
 n^{1/p-\beta}.
\end{equation}
Consequently, the dimensional power in the first line of
\eqref{eq:intro-known-dimensional-upper} is optimal.
\end{prop}

On the other hand, a dimension-free estimate due to Naor was recorded, with a new proof, by Eskenazis and Ivanisvili \cite[Proposition 43]{E-I2020}.  More precisely, for every
$1<p<2$ and every $0<\varepsilon<1/2$, there exists a constant $C_{p}>0$ depending only on $p$, such that
\begin{equation}\label{eq:epsilon-loss-riesz}
 \|\nabla f\|_{L_{p}(\Omega_{n};\ell_{2}^{n})}
 \leq\left(\frac{C_{p}}{\varepsilon}\right)\left\|\Delta^{\frac{1}{p}+\varepsilon}f\right\|_{L_{p}(\Omega_{n})}
\end{equation}
for every scalar-valued function $f$ on $\Omega_{n}$. A later Banach space-valued result of Ivanisvili and Volberg \cite[Theorem 4.1]{I-V2022}, combined with the duality argument, also implies an estimate of the form \eqref{eq:epsilon-loss-riesz} with an unspecified constant $C_{p,\varepsilon}$ for $0<\varepsilon<(p-1)/p$. 

At the critical power $\beta=1/p$, however, the Riesz-product obstruction in \eqref{eq:intro-known-dimensional-lower} becomes dimension-independent,
while the previously available upper bound retains the factor
$1+\log n$. Ivanisvili and Volberg \cite{I-V2022} therefore posed the dimension-free endpoint
\[
 \sup_{n\geq1}\mathfrak R_{p,1/p}(n)<\infty
\]
as an open problem \cite[Section 8.1]{I-V2022} (we also refer the reader to \cite{ChenDai2026,DIPV2026} for some recent developments  about Riesz transforms on the hypercube).

In this paper, we introduce a probabilistic representation that preserves the
cancellation structure lost in the classical semigroup approach.  This allows us
to remove the $\varepsilon$-loss in the previously known estimate \eqref{eq:epsilon-loss-riesz},
improving the power $\Delta^{1/p+\varepsilon}$ to the critical power
$\Delta^{1/p}$ and thereby resolving the corresponding open problem.  Our first main theorem is now stated as follows.

\begin{thm}\label{thm:main}
For every $1<p\leq 2,$ for every
$n\in\N$ and every $f:\Omega_n\to\mathbb{C},$ we have
\[
\left\|\left(\sum_{j=1}^n|D_jf|^2\right)^{\frac12}\right\|_{L_p(\Omega_n)} \leq c_{{\rm abs}}(p-1)^{-2}\|\Delta^{\frac1p}f\|_{L_p(\Omega_n)}.
\]
\end{thm}

One of the principal novelties in our proof lies in the construction of a probabilistic representation of the dual fractional Riesz transform.  To explain the necessity of this construction, we first recall the limitation of the classical semigroup approach.  The standard heat semigroup representation
\[
\Delta^{-z}=\frac{1}{\Gamma(z)}\int_{0}^{\infty}t^{z-1}e^{-t\Delta}\,dt
\]
leads to the formula
\[
\mathcal B_z h=\frac{1}{\Gamma(z)}\int_{0}^{\infty}t^{z-1}e^{-t\Delta}\left(\sum_{j=1}^{n}D_jh_j\right)\,dt .
\]
A direct estimate based on the classical heat semigroup representation
proceeds by applying Minkowski's inequality to the singular integral.
At this stage, the argument no longer retains the cancellation among the
coordinate contributions appearing in each Walsh coefficient. Indeed, for
every nonempty $A\subseteq[n]$,
\[
\widehat{\mathcal B_{z}h}(A)
=
|A|^{-z}\sum_{j\in A}\widehat h_{j}(A),
\]
so the relevant cancellation is contained in the signed sum $\sum_{j\in A}\widehat h_{j}(A)$ before absolute values or $L_{p}$ norms are taken. At the critical exponent $z=1/p$, the resulting loss of
cancellation leads to the logarithmic divergence of the corresponding
integral.

We overcome this obstruction by introducing a probabilistic representation invoking two exponential clock systems. The construction is inspired by the martingale approach to singular integrals, where endpoint estimates are naturally formulated in terms of martingale BMO rather than $L_\infty$.  Moreover, the two independent exponential clock systems are assigned distinct roles in the representation. The coordinate-selection clock system determines the random ordering of the coordinates and hence the associated martingale-difference projections. The spectral-scaling clock system contributes the complementary spectral factor required in the dual fractional Riesz multiplier. Since
the present lifting procedure a two-parameter martingale structure, it is natural to seek a product BMO endpoint estimate.  More precisely, in Section \ref{sec:lift}, we first reduce Theorem \ref{thm:main} to studying the boundedness of $\mathcal{A}_{n,z}$ with $\Re(z)\in [1/2,1]$; see Proposition \ref{thm:averaging} for details. Later, the Section \ref{endpoint estimate section} is devoted to establishing the endpoint estimate for $\mathcal{A}_{n,z}$. After the localization technique applying to the family $\{\mathcal{A}_{n,z}\}_{\Re(z)\in [1/2,1]}$, we obtain the continuity and analyticity of the family which enables us to apply the Stein interpolation theory to establish Theorem \ref{thm:main}.

Furthermore, the first-order estimate admits a natural higher-order extension.  At the level of the probabilistic representation, one may formulate a direct higher-order analogue of the two-clock lifting: for an ordered multi-index $\mathbf j=(j_{1},\ldots,j_{k})\in[n]^{k}$ one introduces $k$ coordinate-selection exponential clock systems, corresponding to the $k$ discrete derivatives, together with one
additional exponential clock system that supplies the remaining
fractional spectral factor. Establishing $(k+1)$-parameter product-BMO endpoint for the analytic family of lifting operators and applying the multiparameter interpolation theory will yield the desired estimate.

For the purposes of the present paper, a more efficient proof is obtained
from a dimension-free Hilbert-valued amplification of the first-order
estimate. 
Since the operators $D_{j}$ commute with one another and with
every fractional power of $\Delta$, the Hilbert-valued first-order
estimate can be applied successively to the vector fields produced at
the preceding stages. For an integer $k\geq1$ and an ordered multi-index
$\mathbf j=(j_{1},\ldots,j_{k})\in[n]^{k}$, put
\[
 D_{\mathbf j}=D_{j_{1}}\cdots D_{j_{k}},
 \qquad
 \nabla^kf=\{D_{\mathbf j}f\}_{\mathbf j\in[n]^{k}}.
\]
Our second main result, the higher-order sharp fractional Riesz estimate, is stated as follows.

\begin{thm}\label{thm:higher-main}
Let $k\geq1$ be an integer and $1<p<2.$ For every $n\in\mathbb{N}$ and for every $f:\Omega_n\to\mathbb{C},$ we have
\[
 \left\|\left(\sum_{j_1,\cdots,j_k=1}^n
 |D_{j_{1}}\cdots D_{j_{k}}f|^{2}\right)^{\frac12}\right\|_{L_p(\Omega_n)}\leq c_k(p-1)^{-2k}\|\Delta^{\frac{k}{p}}f\|_{L_p(\Omega_n)}.
\]
The exponent $\frac{k}{p}$ is optimal even if the sum on the left is restricted to pairwise distinct indices.
\end{thm}

\begin{rem}
We prove in Section \ref{optimality on constants} that 
\[
\left\|\nabla^k\Delta^{-\frac{k}{p}}\right\|_{L_p(\Omega_n)\to L_p(\Omega_n,\ell_2^{n^{k}})}\geq c_k(p-1)^{-\frac{k}{2}}
\]
The gap between the upper bound $c_k(p-1)^{-2k}$ and the lower bound $c_k(p-1)^{-\frac{k}{2}}$ is left open.
\end{rem}

%

Our paper is organized as follows. Section \ref{sec:preliminaries} contains preliminary material. Section \ref{sec:lift} provides the transference principle laying at the core of our approach. In Section \ref{sec:lift}, we introduce an (analytic) family $\mathcal{A}_{n,z}$ of operators whose boundedness (in suitable space) constitutes the bulk of the proof. The $L_2$-boundedness on the line $\{\Re(z)=\frac12\}$ of $\mathcal{A}_{n,z}$  is established in Subsection \ref{sec:left}.  The $L_{\infty}-\BMO$ estimate on the line $\{\Re(z)=1\}$  of $\mathcal{A}_{n,z}$  is established in the Subsection \ref{sec:right}.   In Section \ref{sec:regularization}, we establish the desired $L_{q}$-boundedness for the operator $\mathcal{A}_{n,1-\frac1q}$ via the Stein interpolation theorem. In Section \ref{main result section}, we prove Theorems \ref{thm:main} and \ref{thm:higher-main}.  The  sharpness of the exponent, and the optimality of the constants in Theorem \ref{thm:higher-main} are also discussed in Section \ref{optimality on exponent} and  Section \ref{optimality on constants}, respectively. Finally, Section \ref{sec:applications} presents applications of our sharp fractional Riesz estimate, including the optimal short-time gradient estimate of Eskenazis and Ivanisvili \cite{E-I2020}, the  logarithm-free Bernstein--Markov inequality of Volberg \cite{V2024}, 

\begin{rem}
Upon completing our work, we became aware of a very recent paper
by Xu and Zhang \cite{XuZhang}, in which the sharp Riesz estimate is also established. Their
approach, however, relies on the BMO theory for noncommutative semigroups and is thus fundamentally different from the one developed here. Furthermore, our two-clock
method can be adapted to prove a stronger form of the convolution inequality
raised in \cite[Question 6.1]{NS2016}, as well as to yield a new proof of the sharp metric $X_p$
inequality posed by Naor in \cite{Naor2016}--a result recently established by Areshidze \cite{Are2026}.
The detailed proofs of these additional results will be presented in our forthcoming
paper.
\end{rem}

%
\section{Preliminaries}\label{sec:preliminaries}

The notations used throughout this paper are standard. Let $\mathbb{R}$ and $\mathbb{C}$ be the fields of real numbers and complex numbers, respectively. For a natural number $n,$ let $\mathfrak{S}_n$ be the group of all permutations on $\{1,\cdots,n\}$ and $[n]=\{1,\cdots,n\}$. For a linear operator $T:X\to X$, we denote the operator norm by $\|T\|_{X\circlearrowleft}$. We use $c_{\rm abs}$ to denote an absolute constant, whose value may vary from line to line. For a subset $A\subseteq [n]$, we let $|A|$ be the cardinality of $A$.

\subsection{Analysis on the hypercube}
For $n\in\mathbb{N},$ write $\Omega_{n}=\{-1,1\}^{n}.$ Let $\mu$ be a normalised uniform measure on $\{-1,1\}.$  We equip $\Omega_n$ with the measure $\mu_n:=\mu^{\otimes n}.$ Note that
$$\int_{\Omega_n}fd\mu_n=2^{-n}\sum_{x\in\Omega_n}f(x).$$
For $1\leq p\leq\infty$, we use the usual space $L_p(\Omega_n)$. For $x\in\Omega_{n}$ and $j\in\mathbb{N},$ let $x^{(j)}$ be obtained from $x$ by replacing $x_j$ with $-x_j.$ The $j$-th Walsh derivative is given by the formula
$$(D_jf)(x)=\frac{f(x)-f(x^{(j)})}{2}.$$
The operators $\{D_j\}_{j\geq1}$ are commuting self-adjoint projections on
$L_2(\Omega_{n})$ and contractions on every $L_p(\Omega_{n}).$ For $A\subset [n]$, the Walsh function $w_A$ is defined by the formula
$$w_A(x)=\prod_{j\in A}x_j.$$
The Walsh system $\{w_A\}_{A\subset [n]}$ is an orthonormal basis of $L_2(\Omega_n)$ (respectively, of $L_2(\Omega_{n})$). We have
$$\widehat{f}(A)=\langle f,w_A\rangle,\quad f=\sum_{A\subset [n]}\widehat{f}(A)w_A.$$
The Walsh derivatives is defined by $D_jw_{A}=\1_A(j)w_A$. For a subset $A\subseteq [n],$ let $|A|$ be the cardinality of $A.$ The Walsh gradient and the Walsh Laplacian are given by the formulae
$$\nabla f=\{D_jf\}_{j=1}^{n},\quad \Delta=\sum_{j=1}^nD_j.$$
Thus $\Delta w_A=|A|w_A.$ Consequently, $\ker(\nabla)=\ker(\Delta)=\mathbb{C}$. The heat semigroup $e^{-t\Delta}$ acts diagonally by
$$e^{-t\Delta}f=\sum_{A\subset[n]}e^{-t|A|}\widehat{f}(A)w_A,\quad t\geq0.$$
For $z\in\mathbb{C}$, complex power $\Delta^z$ of the Walsh Laplacian on $L_p(\Omega_n)$ is defined by the formula
$$\Delta^zf=\sum_{\emptyset\neq A\subset[n]}|A|^z\widehat{f}(A)w_A.$$

\subsection{Haar system}

We begin this subsection by recalling atoms and Haar function with respect to a dyadic filtration. 
Let $\mathbb{D}_0=\{\Omega_n\}$, and for $1\leq k\leq n$, let 
\[\mathbb{D}_k=\{\{a\}\times \Omega_{n-k}: a\in \Omega_k\},\quad \mathcal{F}_k=\sigma(\mathbb{D}_k).\]
Given $I:=\{a_I\}\times \Omega_{n-k}\in \mathbb{D}_k$ with $0\leq k\leq n-1$, we set 
\[I^+= \{a_I\}\times \{1\}\times\Omega_{n-k-1}\in \mathbb{D}_{k+1},\quad I^-= \{a_I\}\times \{-1\}\times \Omega_{n-k-1}\in \mathbb{D}_{k+1},\]
for the the left-half of $I$, and the right-half of $I$, respectively. Then  the corresponding Haar function is defined as follows
\[
h_I=2^{\frac{k}{2}}\left(\1_{I^+}-\1_{I^{-}}\right).
\]
We set $h_{\emptyset}=1$ for convenience. Set $\mathcal{D}_n=\{\emptyset\}\bigsqcup\bigsqcup_{k=0}^{n-1}\mathbb{D}_k$, and it is clear that for each $I\in \mathcal{D}_n\setminus \{\emptyset\}$,
\[\int_{\Omega_n}h_I d\mu_n =0, \quad \|h_I\|_{L_2(\Omega_n)}=1.\]

\noindent The family of dyadic rectangles in $\Omega_n\times \Omega_n$ is given by
\[
\mathcal{R}_n=\mathcal{D}_n\otimes \mathcal{D}_n.
\]
For each rectangle $R=I\times J\in \mathcal{R}_{n}$, the Haar function $h_R$ is defined by 
\[
h_R=h_I\otimes h_J.
\]

\begin{lem}\label{orthonormal basis lemma}
The family of all Haar functions $\{h_I\}_{I\in \mathcal{D}_n}$ (respectively, $\{h_R\}_{R\in \mathcal{R}_n}$) forms an orthonormal basis of $L_2(\Omega_n)$ (respectively, $L_2(\Omega_n\times \Omega_n)$). 
\end{lem} 

We conclude this section by introducing the dyadic product BMO space on $\Omega_{n}\times \Omega_{n}$. The following BMO space was introduced by Bernard \cite{Bernard}. It is the dyadic two-parameter martingale BMO space (see e.g. \cite{We1994}), also the dyadic case of the product BMO in the sense of Chang and Fefferman \cite{CF1980}.
For each $U\subset \Omega_n\times \Omega_n$, let $P_{U}$ be the orthogonal projection on the subspace spanned by the Haar functions $h_R$, $R\in \mathcal{R}_{n}$ and $R\subset U$. Define
\begin{align*}
	\|f\|_{\mathrm{BMO}}&=\sup_{U\subset  \Omega_n\times \Omega_n}\frac{1}{|U|^{1/2}}\|P_{U}f\|_{L_2}\\
	&=\sup_{U\subset  \Omega_n\times \Omega_n}\frac{1}{|U|^{1/2}} \Big(\sum_{R\in \mathcal{R}_{n}, R\subset U} |\langle f, h_R\rangle|^2\Big)^{1/2}.
\end{align*}

\subsection{Complex interpolation}
 In  subsection, we collect several complex interpolation results which will be used later.
The following theorem is taken from \cite[Theorem 5.1.2]{BL1976}.

\begin{thm}\label{thm:bochner}
Let $(S,m)$ be a probability space and let $E_{0}$ and $E_{1}$ be a compatible Banach couple. If $q\geq 2$ and $\theta=1-\frac{2}{q},$ then
\begin{equation*}
 [L_{2}(S;E_{0}),L_{\infty}(S;E_{1})]_{\theta}
 =L_q(S;[E_{0},E_{1}]_{\theta})
\end{equation*}
isometrically.
\end{thm}

We use the following basic property of complex interpolation.

\begin{thm}\label{standard complex interpolation theorem}
Let $X_1$ be a Banach space such that $X_{1}$ embeds into $X_{0}$ continuously. Consider the family of bounded linear operators $F(z):X_0\to Y_0,$ $0\leq\Re(z)\leq 1.$ Suppose that $z\to F(z)h$ is $Y_0$-analytic on the strip $\{0<\Re(z)<1\}$ for every $h\in X_0.$ Suppose that the map $z\to F(z)h$ is $Y_0\cap Y_1$-continuous on the closed strip $\{0\leq\Re(z)\leq 1\}$ for every $h\in X_0.$ It follows that
$$
\|F(\theta)\|_{[X_0,X_1]_{\theta}\to [Y_0,Y_1]_{\theta}}\leq\max\left\{\sup_{\Re(z)=0}\|F(z)\|_{X_0\to Y_0},\sup_{\Re(z)=1}\|F(z)\|_{X_1\to Y_1}\right\}.
$$
\end{thm}
\noindent 
Let \((X_0,X_1)\) be a compatible Banach couple, and let \(0<\theta<1\). We define the norm on \(X_0^{1-\theta}X_1^\theta\) as follows:
\begin{equation*}
\|a\|_{X_0^{1-\theta}X_1^{\theta}}=\inf\left\{\|b\|_{X_0}^{1-\theta}\|c\|_{X_1}^{\theta}:\ |a|\leq |b|^{1-\theta}|c|^{\theta}\right\}.
\end{equation*}
The following result is standard (see e.g. Theorem 4.6 in \cite{Kalton-MS}). The Radon-Nikodym property is also referred to \cite{Kalton-MS}. 

\begin{thm}\label{lem:lattice}
Let $X_0$ and $X_1$ be Banach lattices. Suppose that either $X_0$ or $X_1$ possesses Radon-Nikodym property. It follows that
\begin{equation*}
 [X_0,X_1]_{\theta}=X_0^{1-\theta}X_1^{\theta}
\end{equation*}
isometrically.
\end{thm}

We conclude this subsection with the following interpolation result which is one of the key ingredients for proving Theorem \ref{thm:main}. Some versions of the following result might be known to specialists. However, the precise statement below is not given anywhere in the literature. The most important here is the constant $q^2.$ The proof is provided in Appendix \ref{sec:product-interpolation}. 

\begin{thm}\label{thm:product-interpolation}
Let $2\leq q<\infty$ and $\theta=1-\frac{2}{q}.$ We have
\[
\|\cdot\|_{L_q(\Omega_n\times\Omega_n)}\leq c_{{\rm abs}}q^2\|\cdot\|_{[L_2(\Omega_n\times\Omega_n),\BMO(\Omega_n\times\Omega_n)]_{\theta}}.
\]
\end{thm}

\section{Reduction of Theorem \ref{thm:main} via probabilistic representation }\label{sec:lift}

In this section, we first reduce Theorem \ref{thm:main} to establishing the boundedness of the operator $\mathcal{B}_{n,\frac1p}.$ After that, using a probabilistic representation of $\mathcal{B}_{n,z}$ via $\mathcal{A}_{n,z},$  we further reduce  Theorem \ref{thm:main} to studying the boundedness of $\mathcal{A}_{n,\frac1p}.$
 
\subsection{Reduction to the family $\mathcal{B}_{n,z}$}

For $z\in \mathbb{C}$ with $1/2\leq\Re(z)\leq1$, we define
\begin{equation*}
\mathcal{B}_{n,z}h=\Delta^{-z}\left(\sum_{j=1}^{n}D_{j}h_{j}\right),
\end{equation*}
where $h=(h_{j})_{j=1}^{n}\in L_{\infty}(\Omega_{n};\ell_{2}^{n})$. It is clear that for each $A\subset  [n]$,  
\begin{equation*}
\widehat{\mathcal{B}_{n,z}h}(A)=|A|^{-z}\sum_{j\in A}\widehat{h_j}(A),\quad A\neq\emptyset.
\end{equation*}

\begin{lem}\label{lem:dual-reduction} Let $1\leq p\leq\infty.$ The dual of $\nabla\Delta^{-z}:L_p(\Omega_n)\to L_p(\Omega_n,\ell^{n}_{2})$ is $\mathcal{B}_{n,z}:L_{\frac{p}{p-1}}(\Omega_n;\ell^{n}_{2})\to L_{\frac{p}{p-1}}(\Omega_n).$
\end{lem}
\begin{proof} Self-adjointness and commutation of $D_{j}$ and $\Delta^{-z}$ give
$$
\langle \nabla\Delta^{-z}g,h\rangle_{L_{2}(\ell_{2})}=\sum_{j=1}^{n}\langle D_{j}\Delta^{-z}g,h_{j}\rangle=\left\langle g,\Delta^{-z}\sum_{j=1}^{n}D_{j}h_{j}\right\rangle.
$$
\end{proof}

Due to Lemma \ref{lem:dual-reduction}, the proof of Theorem \ref{thm:main} is equivalent to establish the boundedness of $\mathcal{B}_{n,\frac{1}{p}}$ from $L_{\frac{p}{p-1}}(\Omega_n;\ell^{n}_{2})$ to $L_{\frac{p}{p-1}}(\Omega_n)$.

\subsection{The family $\mathcal{A}_{n,z}$}

For $A\subset [n]$ and $j\in [n],$ set
$$V_{A,j}=\{\tau\in\mathbb{R}^n_+:\ \tau_i\geq \tau_j,\quad i\in A\}.$$
For a given $\tau\in \mathbb{R}^{n}_{+}$ and $1\leq j\leq n,$ define a projection $P_j^{\tau}$ on $L_2(\Omega_n)$ by setting
$$P_j^{\tau}w_A=\1_A(j)\1_{V_{A,j}}(\tau)w_A,\quad A\subset  [n].$$
The family $\{\mathcal{A}_{n,z}\}_{z\in \mathbb{C}}$ is defined as follows. For $z\in\mathbb{C},$ let
\begin{equation*}
(\mathcal{A}_{n,z}h)(\tau,\sigma)
 =\sum_{j,k=1}^{n}
 \tau_j^{\frac{z-1}{2}}\sigma_k^{\frac{z-1}{2}}
 (P_j^\tau\otimes P_k^{\sigma})(\mathcal{J}h_j),\quad \tau,\sigma\in\mathbb{R}^n_+,
\end{equation*}
where $h:\Omega_n\to \ell ^n_2$ and $h=(h_1,\cdots,h_n)$. Moreover, the isometric lifting operator $\mathcal{J}:L_1(\Omega_{n})\to L_1(\Omega_{n}\times \Omega_{n})$ is defined by
\begin{equation*}
(\mathcal{J}f)(x,y)\coloneqq f(x_{1}y_{1},\cdots, x_{n}y_{n}),
\end{equation*}
for all $f:\Omega_{n}\to \mathbb{C}$ and $x=(x_{1},\cdots,x_{n})$, $y=(y_{1},\cdots, y_{n})\in \Omega_{n}$. The following lemma is immediately from the definition.
\begin{lem}\label{lem:J}
For every $1\leq r\leq\infty$, we have
\begin{equation*}
 \|\mathcal{J}f\|_{L_{r}(\Omega_{n}\times \Omega_{n})}=\|f\|_{L_{r}(\Omega_{n})}.
\end{equation*}
Moreover, for each $A\subset [n],$
$$\mathcal{J}w_A=w_A\otimes w_A,\qquad (D_j\otimes{\rm id})(\mathcal{J}f)=({\rm id}\otimes D_j)(\mathcal{J}f)=\mathcal{J}(D_jf).$$
\end{lem}

\subsection{Probabilistic representation}

In order to provide a probabilistic representation of the operator $\mathcal{B}_{n,z}$, we introduce specific probability measure on a certain ``clock space". We equip $\mathbb{R}_+$ with a probability measure $\nu$ given by $d\nu=e^{-t}dt$ and equip $\mathbb{R}^n_+$ with the product measure $\nu^{\otimes n}$ making $(\mathbb{R}^n_+,\nu^{\otimes n})$ a probability space. The ``clock space" is now defined by
$${\rm Clock}_n=(\mathbb{R}^n_+\times\mathbb{R}^n_+,\nu^{\otimes n}\otimes\nu^{\otimes n}).$$

Our probabilistic representation of the operator $\mathcal{B}_{n,z}$ is given below. This result is one of the main contributions of the present paper and serves as a crucial bridge in the proof.
\begin{prop}\label{thm:averaging}
For every $z$ in the closed strip $\{z\in \mathbb{C}:\Re(z)\in [1/2,1]\}$, we have
\begin{equation*}
\mathbb{E}_{\nu^{\otimes n}\otimes \nu^{\otimes n}}\circ\mathcal{A}_{n,z}=\Gamma^2\left(\frac{z+1}{2}\right)\mathcal{J}\circ\mathcal{B}_{n,z},
\end{equation*}
where $\mathbb{E}_{\nu^{\otimes n}\otimes \nu^{\otimes n}}$ stands for the expectation taken with respect to $\nu^{\otimes n}\otimes \nu^{\otimes n}$.
\end{prop}
\begin{proof}
It suffices to prove the equality on the function
$$h=
\begin{pmatrix}
\underbrace{0,\cdots,0,}_{(l-1)\mbox{ times}} & w_A, & \underbrace{0,\cdots,0}_{(n-l)\mbox{ times}}
\end{pmatrix},\quad A\subset [n].$$
In this case, we have
$$(\mathcal{A}_{n,z}h)(\tau,\sigma)=\sum_{k=1}^n\tau_l^{\frac{z-1}{2}}\sigma_k^{\frac{z-1}{2}}
 (P_l^{\tau}\otimes P_k^{\sigma})(\mathcal{J}w_A),\quad \tau,\sigma\in\mathbb{R}^n_+.$$
By Lemma \ref{lem:J}, $\mathcal{J}w_A=w_A\otimes w_A.$ This allows us to write
$$(P_l^{\tau}\otimes P_k^{\sigma})(\mathcal{J} w_A)=P_l^{\tau}w_A\otimes P_k^{\sigma}w_A=\1_A(l)\1_A(k)\1_{V_{A,l}}(\tau)\1_{V_{A,k}}(\sigma)w_A\otimes w_A.$$
We now have
$$(\mathcal{A}_{n,z}h)(\tau,\sigma)
 =\sum_{k=1}^n
 \tau_l^{\frac{z-1}{2}}\sigma_k^{\frac{z-1}{2}}\1_A(l)\1_A(k)\1_{V_{A,l}}(\tau)\1_{V_{A,k}}(\sigma)
 w_A\otimes w_A.$$
Thus,
\begin{align*}
&\mathbb{E}_{\nu^{\otimes n}\otimes \nu^{\otimes n}}(\mathcal{A}_{n,z}h)\\
=&\sum_{k=1}^n\left(\1_A(l)\int_{V_{A,l}}\tau_l^{\frac{z-1}{2}}\prod_{m=1}^ne^{-\tau_m}d\tau_m\right)\\
&\times\left(\1_A(k)\int_{V_{A,k}}\sigma_k^{\frac{z-1}{2}}\prod_{m=1}^ne^{-\sigma_m}d\sigma_m\right)\times (w_A\otimes w_A).
\end{align*}
It follows from the Fubini theorem that for every $l\in A$ we have
\begin{align*}
\int_{V_{A,l}}\tau_l^{\frac{z-1}{2}}\prod_{m=1}^ne^{-\tau_m}d\tau_m&=\int_{V_{A,l}}\tau_l^{\frac{z-1}{2}}\prod_{m=1}^ne^{-\tau_m}d\tau_m\\
&=\int_0^{\infty}u^{\frac{z-1}{2}}\Big(\int_u^{\infty}e^{-v}dv\Big)^{|A|-1}e^{-u}du\\
&=\int_0^{\infty}u^{\frac{z-1}{2}}e^{-|A|u}du=|A|^{-\frac{z+1}{2}}\Gamma\left(\frac{z+1}{2}\right).
\end{align*}
Hence,
\begin{align*}
\mathbb{E}_{\nu^{\otimes n}\otimes \nu^{\otimes n}}(\mathcal{A}_{n,z}h)&=\sum_{k=1}^n\1_A(l)\cdot \1_A(k)|A|^{-(z+1)}\Gamma\left(\frac{z+1}{2}\right)^{2}\cdot w_A\otimes w_A\\
&=\1_A(l)|A|^{-z}\cdot w_A\otimes w_A.
\end{align*}

On the other hand, since the function has the following concrete form
$$
h=
\begin{pmatrix}
\underbrace{0,\cdots,0,}_{(l-1)\mbox{ times}} & w_A, & \underbrace{0,\cdots,0}_{(n-l)\mbox{ times}}
\end{pmatrix},
$$
it gives that
$$\mathcal{B}_{n,z}h=\Delta^{-z}\left(\sum_{j=1}^nD_jh_j\right)=\Delta^{-z}D_lw_A=\Delta^{-z}\1_A(l)w_A=|A|^{-z}\1_A(l)w_A.$$
Thus,
$$J(\mathcal{B}_{n,z}h)=|A|^{-z}\1_A(l)\cdot w_A\otimes w_A.$$
Comparing this with the preceding paragraph, we complete the proof.
\end{proof}

\begin{rem}
Thanks to Proposition \ref{thm:averaging} and Lemma \ref{lem:dual-reduction}, the boundedness of $\mathcal{B}_{n,\frac1p}:L_p(\Omega_n,\ell_2^n)\to L_p(\Omega_n)$ (and, hence, Theorem \ref{thm:main}) reduces to the boundedness of $\mathcal{A}_{n,\frac1p}:L_{\frac{p}{p-1}}(\Omega_n,\ell_2^n)\to L_{\frac{p}{p-1}}({\rm Clock}_n\times\Omega_n\times\Omega_n).$
\end{rem}

\begin{rem}
We now explain why there are two instance of $\mathbb{R}^{n}_{+}$ in the ${\rm Clock}_{n}$. Suppose, we take ${\rm Clock}_n=(\mathbb{R}^n_+,\nu^{\otimes n})$ and set
$$(\mathcal{A}_{n,z}h)(\tau)=\sum_{k=1}^n\tau_k^{z-1}P_k^{\tau}h_k.$$
The analogue of Theorem \ref{thm:left} would be
$$\|\mathcal{A}_{n,z}h\|_{L_2({\rm Clock}_n;L_2(\Omega_n))}\leq c_{{\rm abs}}\|h\|_{L_2(\Omega_n;\ell_2^n)},\quad \Re(z)=\frac12.$$
This is, however, not the case. In fact, the left hand side is infinite. The reason is that, when repeating the argument in Theorem \ref{thm:left}, then, instead of the factor
$$\Big(\int_0^{\infty}u^{-\frac12}e^{-u}du\Big)^2,$$
we obtain the factor $\int_0^{\infty}u^{-1}e^{-u}du$ which is infinite.
\end{rem}

\section{Endpoint estimates for $\mathcal{A}_{n,z}$}\label{endpoint estimate section}

We now establish the endpoint estimates for the operator $\mathcal{A}_{n,z}$, which will be needed for the interpolation argument in the next section.
In Section \ref{sec:left}, we show the boundedness  $\mathcal{A}_{n,z}: L_2(\Omega_n; \ell_2^n)\to L_2({\rm Clock}_n;L_2(\Omega_n\times\Omega_n)).$ In Section \ref{sec:right}, we establish the key ingredient which provides the boundedness of $\mathcal{A}_{n,z}(\tau,\sigma)$ from $L_{\infty}(\Omega_n; \ell_2^n)$ to $\BMO(\Omega_n\times\Omega_n)$ for any given $(\tau,\sigma)\in{\rm Clock}_n.$

\subsection{The boundedness of $\mathcal{A}_{z}$ on Hilbert space}\label{sec:left}
\begin{thm}\label{thm:left}
For every $f=(f_1,\cdots,f_n)\in L_2(\Omega_n;\ell^n_2)$, we have
\begin{equation*}
\|\mathcal{A}_{n,z}f\|_{L_2({\rm Clock}_n;L_2(\Omega_n\times\Omega_n))}\leq\pi^{\frac12}\|f\|_{L_2(\Omega_n;\ell_2^n)},\quad \Re(z)=\frac12.
\end{equation*}
\end{thm}

\begin{proof} 
For every $f=(f_1,\cdots,f_n)\in L_2(\Omega_n;\ell^n_2)$,	we have
\begin{align*}
&\|\mathcal{A}_{n,z}f\|_{L_2({\rm Clock}_n;L_2(\Omega_n\times\Omega_n)}^2\\
=&\int_{\mathbb{R}^{n}_{+}\times \mathbb{R}^{n}_{+}}\|(\mathcal{A}_{n,z}f)(\tau,\sigma)\|_{L_2(\Omega_n\times\Omega_n)}^2\prod_{m=1}^ne^{-\tau_m-\sigma_m}d\tau_md\sigma_m.
\end{align*}

For any fixed $\tau,\sigma\in\mathbb{R}^n_+,$ the projections
$P_j^{\tau}\otimes P_k^{\sigma}$ have pairwise orthogonal ranges. Hence,
$$
\|(\mathcal{A}_{n,z}f)(\tau,\sigma)\|_{L_2(\Omega_n\times\Omega_n)}^2=\sum_{j,k=1}^n\left|\tau_j^{\frac{z-1}{2}}\cdot\sigma_k^{\frac{z-1}{2}}\right|^2
 \|(P_j^{\tau}\otimes P_k^{\sigma})(\mathcal{J}f_j)\|_{L_2(\Omega_n\times\Omega_n)}^2.
$$
For each $j\in  [n]$, it follows  from the Walsh expansion of $h_j$ that 
$$
\mathcal{J}(f_j)=\sum_{A\subset [n]} \widehat{f_j}(A)\cdot (w_A\otimes w_A)
$$
and
\begin{align*}
(P_j^{\tau}\otimes P_k^{\sigma})(\mathcal{J}f_{j})&=\sum_{A\subset  [n]}\widehat{f_j}(A)(P_j^{\tau}w_A\otimes P_k^{\sigma}w_A)\\
&=\sum_{A\subset [n]} \1_A(j)\1_A(k)\1_{V_{A,j}}(\tau)\1_{V_{A,k}}(\sigma)\widehat{f_j}(A)\cdot (w_A\otimes w_A)
\end{align*}
The summands on the right hand side are pairwise orthogonal in $L_2(\Omega_n\times\Omega_n).$ Therefore,
$$\|(P_j^{\tau}\otimes P_k^{\sigma})(\mathcal{J}f_j)\|_{L_2(\Omega_n\times\Omega_n)}^2=\sum_{A\subset [n]} |\widehat{f_j}(A)|^{2} \1_A(j)\1_A(k)\1_{V_{A,j}}(\tau)\1_{V_{A,k}}(\sigma).$$

Note that $\Re(z)=\frac12$. Combining all the preceding equalities, we have
\begin{align*}
&\|\mathcal{A}_{n,z}f\|_{L_2({\rm Clock}_n;L_2(\Omega_n\times\Omega_n))}^2\\
=&\sum_{A\subset [n]} \sum_{j,k=1}^n|\widehat{f_j}(A)|^2\1_A(j)\1_A(k)\left(\int_{V_{A,j}}\int_{V_{A,k}}\tau_j^{-\frac12}\sigma_k^{-\frac12}\prod_{m=1}^ne^{-\tau_m-\sigma_m}d\tau_md\sigma_m\right).
\end{align*}
Fubini theorem yields
\begin{align*}
&\int_{V_{A,j}}\int_{V_{A,k}}\tau_j^{-\frac12}\sigma_k^{-\frac12}\prod_{m=1}^ne^{-\tau_m-\sigma_m}d\tau_md\sigma_m\\
=&\left(\int_0^{\infty}u^{-\frac12}\left(\int_u^{\infty}e^{-v}dv\right)^{|A|-1}e^{-u}du\right)^2=\left(\int_0^{\infty}u^{-\frac12}e^{-|A|u}du\right)^2=\pi|A|^{-1},
\end{align*}
for all $j$, $k\in A$. Finally,
\begin{align*}
&\|\mathcal{A}_{n,z}f\|_{L_2({\rm Clock}_n;L_2(\Omega_n\times\Omega_n))}^2\\
=&\sum_{\emptyset \neq A \subset [n]}\sum_{j,k\in A}|\widehat{f_j}(A)|^2\cdot \pi|A|^{-1}=\pi\sum_{\emptyset \neq A \subset [n]}\sum_{j\in A}|\widehat{f_j}(A)|^2\\
 \leq&\pi\sum_{A\subset [n]}\sum_{j=1}^n|\widehat{h_j}(A)|^2\leq\pi\sum_{j=1}^n\sum_{A\subset [n]} |\widehat{f_j}(A)|^2\\
=&\pi\sum_{j=1}^n\|f_j\|_{L_2(\Omega_n)}^2=\pi\|f\|_{L_2(\Omega_n,\ell^{n}_{2})}^2.
\end{align*}
The proof is complete.
\end{proof}

\subsection{The  $L_{\infty}-\mathrm{BMO}$ estimate for  $\mathcal{A}_{n,z}$ }\label{sec:right}

The following lemma provides an equivalent representation of the projection operator $P^{\tau}_{j}$, which will be useful in the subsequent proofs. For $f:\Omega_n\to \mathbb{C}$ and a permutation $\pi\in \mathfrak{S}_{n},$ the function $f\circ \pi$ is given by
$$(f\circ\pi)(x_{1},\dots,x_{n})\coloneqq f(x_{\pi(1)},\dots, x_{\pi(n)}),\quad (x_{1},\dots, x_{n})\in \Omega_n.$$
\begin{lem}\label{main sijie lemma}
Fix a permutation $\pi\in\mathfrak{S}_n$. For every $\tau\in \mathbb{R}^n_{+}$ such that
\begin{equation}\label{tau vs pi eq}
\tau_{\pi(1)}>\tau_{\pi(2)}>\cdots>\tau_{\pi(n)},
\end{equation}
we have $P_j^{\tau}(h_I\circ\pi)=\delta_{j,\pi(k+1)}h_I\circ\pi$, for every $I\in\mathbb{D}_k,$ $1\leq k<n,$ and for every $j\in [n].$
\end{lem}
\begin{proof} 
	Fix  $1\leq k \leq n-1$ and $j\in [n]$. 
	Since $I\in\mathbb{D}_k$, we may write $I=\{a_I\}\times \Omega_{n-k}$, $a_I\in \Omega_k$.	
Write
$$a_I=\{y_1,\cdots,y_k\}\in\Omega_k.$$
For each $x\in \Omega_n$, we have
$$h_I(x)=2^{-\frac{k}{2}}x_{k+1}\prod_{m=1}^k(1+y_mx_m)=2^{-\frac{k}{2}}\sum_{A\subset[k]}x_{k+1}\prod_{m\in A}y_mx_m.$$
In other words,
$$h_I=2^{-\frac{k}{2}}\sum_{A\subset [k]}w_A(y)\cdot w_{A\cup\{k+1\}}.$$
Note that $w_B\circ\pi=w_{\pi(B)}$ for every $B\subset [n].$ Thus,
$$h_I\circ\pi=2^{-\frac{k}{2}}\sum_{A\subset[k]}w_A(y)\cdot w_{\pi(A\cup\{k+1\})}.$$

\noindent Recall that for each $j\in [n]$ and for every $B\subset  [n]$,
$$
P^{\tau}_j(w_B)=\1_B(j)\1_{V_{B,j}}(\tau)w_B.
$$
Thus,
\begin{equation}\label{ph expression}
P_j^{\tau}(h_I\circ\pi)=\sum_{A\subset[k]}w_A(y)\cdot \1_{\pi(A\cup\{k+1\})}(j) \1_{V_{\pi(A\cup\{k+1\}),j}}(\tau) w_{\pi(A\cup\{k+1\})}.
\end{equation}

We now claim that for every $1\leq j\leq n$ and for every $A\subset [k]$ the following holds
\begin{equation}\label{key equality 1}
\1_{\pi(A\cup\{k+1\})}(j) \1_{V_{\pi(A\cup\{k+1\}),j}}(\tau)=\delta_{j,\pi(k+1)}.
\end{equation}
Indeed, if the left hand side is non-zero, then $j\in\pi(A)\cup\{\pi(k+1)\}.$ Thus, $j=\pi(m)$ for some $1\leq m\leq k+1.$ On the other hand, we have $\tau\in V_{\pi(A\cup\{k+1\}),j},$ which means
$$\tau_i\geq \tau_j,\quad i\in \pi(A\cup\{k+1\}).$$
In particular, $\tau_{\pi(k+1)}\geq\tau_{\pi(m)}.$ Using \eqref{tau vs pi eq}, we obtain $k+1\leq m.$ Hence, $m=k+1$ and, therefore, $j=\pi(k+1)$ so that $\delta_{j,\pi(k+1)}=1.$ 
Conversely, let $\delta_{j,\pi(k+1)}=1$ so that $j=\pi(k+1).$ Hence, $\1_{\pi(A\cup\{k+1\})}(j)=1.$ We have
$$V_{\pi(A\cup\{k+1\}),j}=V_{\pi(A\cup\{k+1\}),\pi(k+1)}=\{\sigma\in\mathbb{R}^n_+:\ \sigma_{\pi(i)}\geq\sigma_{\pi(k+1)},\quad i\in A\cup\{k+1\}\}.$$
Using \eqref{tau vs pi eq}, we conclude that $\tau$ belongs to the set on the right hand side and, therefore, 
$\1_{V_{\pi(A\cup\{k+1\}),j}}(\tau)=1.$ Hence, the left-hand side is non-zero. 

Using \eqref{key equality 1} and \eqref{ph expression}, we conclude that
$$P_j^{\tau}(h_I\circ\pi)=\delta_{j,\pi(k+1)}\sum_{A\subset [k]}w_A(y)\cdot w_{\pi(A\cup\{k+1\})}=\delta_{j,\pi(k+1)}h_I\circ\pi.$$
This completes the proof.
\end{proof}

Before providing a detailed proof of the main inequality, we apply Lemma \ref{main sijie lemma} to derive the following lemma.

\begin{lem}\label{selector lemma}
Let $f=(f_1, \cdots, f_n)\in L_1(\Omega_n;\ell_2^n)$. Fix two permutations $\pi,\rho\in\mathfrak{S}_n.$ Take $\tau,\sigma\in\mathbb{R}^n_+$ such that
$$\tau_{\pi(1)}>\tau_{\pi(2)}>\cdots>\tau_{\pi(n)},\quad \sigma_{\rho(1)}>\sigma_{\rho(2)}>\cdots>\sigma_{\rho(n)}.$$
If $I,J\in \mathcal{D}_n$, then for $s\in \mathbb{R}$
$$|\langle(\mathcal{A}_{n,1+is}f)(\tau,\sigma),(h_I\circ\pi)\otimes(h_J\circ\rho)\rangle|\leq\max_{1\leq j\leq n}|\langle\mathcal{J}f_j,(h_I\circ\pi)\otimes(h_J\circ\rho)\rangle|.$$
\end{lem}
\begin{proof}
By Lemma \ref{main sijie lemma}, we have the following identity
\begin{align*}
&\langle(\mathcal{A}_{n,1+is}f)(\tau,\sigma),(h_I\circ\pi)\otimes(h_J\circ\rho)\rangle\\
=&\sum_{j_1,j_2=1}^n\tau_{j_1}^{\frac{is}{2}}\sigma_{j_2}^{\frac{is}{2}}\langle (P_{j_1}^{\tau}\otimes P_{j_2}^{\sigma})(\mathcal{J}f_{j_1}),(h_I\circ\pi)\otimes(h_J\circ\rho)\rangle\\
=&\sum_{j_1,j_2=1}^n\tau_{j_1}^{\frac{is}{2}}\sigma_{j_2}^{\frac{is}{2}}\langle \mathcal{J}f_{j_1},P_{j_1}^{\tau}(h_I\circ\pi)\otimes P_{j_2}^{\sigma}(h_J\circ\rho)\rangle.
\end{align*}
If $I = \emptyset$ or $J = \emptyset$, then $P_{j_1}^{\tau}(h_I \circ \pi) = 0$ or $P_{j_2}^{\tau}(h_J \circ \rho) = 0$ for all $1 \leq j_1 \leq n$ and $1 \leq j_2 \leq n$, respectively. If $I\in\mathbb{D}_k$ and $J\in\mathbb{D}_l$ with $1\leq k,l\leq n-1$,  then it follows from the preceding display and Lemma \ref{main sijie lemma} that
\begin{align*}
&\langle(\mathcal{A}_{n,1+is}f)(\tau,\sigma),(h_I\circ\pi)\otimes(h_J\circ\rho)\rangle\\
=&\sum_{j_1,j_2=1}^n\delta_{j_1,\pi( k+1)}\delta_{j_2,\rho(l+1)}\tau_{j_1}^{\frac{is}{2}}\sigma_{j_2}^{\frac{is}{2}}\langle \mathcal{J}f_{j_1},(h_I\circ\pi)\otimes (h_J\circ\rho)\rangle\\
=&\tau_{\pi(k+1)}^{\frac{is}{2}}\sigma_{\rho(l+1)}^{\frac{is}{2}}\langle \mathcal{J}f_{\pi(k+1)},(h_I\circ\pi)\otimes (h_J\circ\rho)\rangle.
\end{align*}
The assertion follows now from the obvious equality $\left|\tau^{\frac{is}{2}}_j\right|=\left|\sigma^{\frac{is}{2}}_j\right|=1.$
\end{proof}

In the sequel, for $U\subset\Omega_n\times\Omega_n$, we denote by $|U|$  the cardinality of of $U$. 
\begin{lem}\label{getting rid of U lemma}
	Let $f=(f_1,\cdots,f_n)\in L_{\infty}(\Omega_n;\ell^n_2).$
	 For every $U\subset\Omega_n\times\Omega_n,$ we have
$$\sum_{j=1}^n\|1_U\cdot \mathcal{J}f_j\|_{L_2(\Omega_n\times\Omega_n)}^2\leq\frac{|U|}{2^{2n}}\|f\|_{L_{\infty}(\Omega_n;\ell^n_2)}^2.$$
\end{lem}
\begin{proof} We have
$$\sum_{j=1}^n\|1_U\cdot \mathcal{J}f_j\|_{L_2(\Omega_n\times\Omega_n)}^2=\int\int_U(\sum_{j=1}^n|\mathcal{J}f_j|^2)d\mu_n d\mu_n.$$
It is immediate that
$$\sum_{j=1}^n|\mathcal{J}f_j|^2=\mathcal{J}F,\quad F=\sum_{j=1}^n|f_j|^2.$$
Thus,
\begin{align*}
\sum_{j=1}^n\|1_U\cdot \mathcal{J}f_j\|_{L_2(\Omega_n\times\Omega_n)}^2&=\int\int_U(\mathcal{J}F)dd\mu_n d\mu_n\\
&\leq  (\mu_n\otimes \mu_n)(U)\|\mathcal{J}F\|_{L_{\infty}(\Omega_n\times\Omega_n)}\\
&=\frac{|U|}{2^{2n}}\|F\|_{L_{\infty}(\Omega_n)}=\frac{|U|}{2^{2n}}\|f\|_{L_{\infty}(\Omega_n;\ell^n_2)}^2.
\end{align*}
\end{proof}

\begin{thm}\label{thm:right}
Fix permutations $\pi$, $\rho\in\mathfrak{S}_n$ and take $\tau,\sigma\in\mathbb{R}^n_+$ such that
$$\tau_{\pi(1)}>\tau_{\pi(2)}>\cdots>\tau_{\pi(n)},\quad \sigma_{\rho(1)}>\sigma_{\rho(2)}>\cdots>\sigma_{\rho(n)}.$$
For  every $f=(f_1,\cdots,f_n)\in L_{\infty}(\Omega_n;\ell^n_2)$, we have
\begin{equation}\label{eq:right-boundary-bmo-bound}
\left\|\left(\mathcal{A}_{n,z}f(\tau,\sigma)\right)\circ(\pi^{-1}\times\rho^{-1})\right\|_{\BMO(\Omega_n\times\Omega_n)}
\leq\|f\|_{L_{\infty}(\Omega_n;\ell^{n}_{2})},\quad \Re(z)=1.
\end{equation}
\end{thm}
\begin{proof} 
	As  $\Re(z)=1$, we may write $z=1+is$, $s\in\mathbb{R}$.  By Lemma \ref{selector lemma}, 
$$|\langle(\mathcal{A}_{n,1+is}f)(\tau,\sigma),(h_I\circ\pi)\otimes(h_J\circ\rho)\rangle|^2\leq\sum_{j=1}^n|\langle \mathcal{J}f_j,(h_I\circ\pi)\otimes(h_J\circ\rho)\rangle|^2.$$

Fix a nonempty subset $U\subset \Omega_n\times\Omega_n.$ We have
\begin{align*}
&\sum_{\substack{I,J\in\mathcal{D}_n\\ I\times J \subset U}}|\langle(\mathcal{A}_{n,1+is}f)(\tau,\sigma),(h_I\circ\pi)\otimes(h_J\circ\rho)\rangle|^2\\
\leq&\sum_{j=1}^n\sum_{\substack{I,J\in\mathcal{D}_n\\ I\times J \subset U}}|\langle \mathcal{J}f_j,(h_I\circ\pi)\otimes(h_J\circ\rho)\rangle|^2\\
=&\sum_{j=1}^n\sum_{\substack{I,J\in\mathcal{D}_n\\ I\times J \subset U}}|\langle(\1_U\circ(\pi\times\rho))\cdot \mathcal{J}f_j,(h_I\circ\pi)\otimes(h_J\circ\rho)\rangle|^2\\
=&\sum_{j=1}^n\sum_{I,J\in{\mathcal{D}}_n}|\langle(\1_U\circ(\pi\times\rho))\cdot \mathcal{J}f_j,(h_I\circ\pi)\otimes(h_J\circ\rho)\rangle|^2\\
=&\sum_{j=1}^n\|(\1_U\circ(\pi\times\rho))\cdot \mathcal{J}f_j\|_{L_2(\Omega_n\times\Omega_n)}^2.
\end{align*}
Applying Lemma \ref{getting rid of U lemma} to the set $(\pi^{-1}\times\rho^{-1})(U),$ we write
$$\sum_{j=1}^n\|(\1_U\circ(\pi\times\rho))\cdot \mathcal{J}f_j\|_{L_2(\Omega_n\times\Omega_n)}^2\leq\frac{|U|}{2^{2n}}\|h\|_{L_{\infty}(\Omega_n;\ell^n_2)}^2.$$
Thus,
$$\sum_{\substack{I,J\in\mathcal{D}_n\\ I\times J \subset U}}|\langle(\mathcal{A}_{n,1+is}f)(\tau,\sigma),(h_I\circ\pi)\otimes(h_J\circ\rho)\rangle|^2\leq\frac{|U|}{2^{2n}}\|f\|_{L_{\infty}(\Omega_n;\ell^n_2)}^2.$$
Taking the supremum over $U$, we complete the proof.
\end{proof}

\section{Boundedness of $\mathcal{A}_{n,1-\frac1q}$ on $L_{q}({\rm Clock}_{n}\times\Omega_{n}\times \Omega_{n})$}\label{sec:regularization}

Based on the results obtained in the previous section, we interpolate the results into the $L_q$ space. The following estimate is the main result of this section.
\begin{thm}\label{key theorem} For $2\leq q<\infty,$ we have
$$\|\mathcal{A}_{n,1-\frac1q}h\|_{L_q({\rm Clock}_n\times\Omega_n\times\Omega_n)}
 \leq c_{{\rm abs}}q^2\|h\|_{L_q(\Omega_n;\ell_2^n)}.$$
\end{thm}

In order to apply the interpolation theory to obtain strong boundedness of $\mathcal{A}_{n,z},$ we first demonstrate the analyticity of the mapping $z\mapsto \mathcal{A}_{n,z}$ for $\Re(z)>0.$ Since $\nu^{\otimes}\otimes \nu^{\otimes n}$ is absolutely continuous with respect to the Lebesgue measure, it follows that for almost every $\tau$, $\sigma\in \mathbb{R}^{n}_{+}$ we have $\tau_{i}\neq \tau_{j}$ and $\sigma_{i}\neq \sigma_{j}$ for all $1\leq i,j\leq n.$ Thus, for almost every $\tau$, $\sigma\in \mathbb{R}^{n}_{+}$ we define the mapping $\mathbf{U}$ by setting
$$
(\mathbf{U}F)(\tau,\sigma)=F(\tau,\sigma)\circ(\pi^{-1}\times\rho^{-1}),
$$
where the permutations $\pi$, $\rho\in \mathfrak{S}_{n}$ are chosen by
\begin{equation}\label{permutation law}
\tau_{\pi(1)}>\tau_{\pi(2)}>\cdots >\tau_{\pi(n)},\quad \sigma_{\rho(1)}>\sigma_{\rho(2)}>\cdots>\sigma_{\rho(n)}.
\end{equation}

\noindent For $\epsilon\in(0,1),$ let
\begin{equation*}
 E_{\epsilon}=\left\{(\tau,\sigma)\in \mathbb{R}^{n}_{+}\times \mathbb{R}^{n}_{+}: \epsilon\leq\tau_j,\sigma_k\leq\epsilon^{-1}\right\}
\end{equation*}
Let $M_f$ denote the multiplication operator defined by $M_fg=fg.$

\begin{lem}\label{final analyticity lemma} For every $f\in L_2(\Omega_n;\ell^{n}_{2}),$ the family
$$
 z\longmapsto M_{\1_{E_{\epsilon}}}\mathbf{U}\mathcal{A}_{n,z}f
$$
is $L_2({\rm Clock}_n;L_2(\Omega_n\times\Omega_n))$-analytic on the half-plane $\{\Re(z)>0\}.$ 
\end{lem}
\begin{proof}
By definition the mapping $M_{\1_{E_{\epsilon}}}\mathbf{U}$ is bounded from the vector-valued space $L_2({\rm Clock}_n;L_2(\Omega_n\times\Omega_n))$ to itself. Therefore, it now remains to establish analyticity of the map
$z\to \mathcal{A}_{n,z}f$ on the half-plane $\{\Re(z)>0\}$. It suffices to prove that derivative $z\to\frac{d}{dz}\left(\mathcal{A}_{n,z}f\right)$ belongs to $L_2({\rm Clock}_n;L_2(\Omega_n\times\Omega_n))$ on the half-plane $\{\Re(z)>0\}.$ 

For almost every $\tau,\sigma\in \mathbb{R}^n_+$ we have
$$\left(\frac{d}{dz}\mathcal{A}_{n,z}f\right)(\tau,\sigma)=\frac12\sum_{j,k=1}^n(\tau_j\sigma_k)^{\frac{z-1}{2}}\log(\tau_j\sigma_k)(P_j^{\tau}\otimes P_k^{\sigma})(\mathcal{J}f_j).$$
Thus,
\begin{equation}\label{Hilbert expansion 1}
\begin{split}
&4\left\|\frac{d}{dz}\mathcal{A}_{n,z}f\right\|_{L_2({\rm Clock}_n;L_2(\Omega_n\times\Omega_n))}^2\\
=&\int_{\mathbb{R}^{n}_{+}\times \mathbb{R}^{n}_{+}}\left\|\sum_{j,k=1}^n(\tau_j\sigma_k)^{\frac{z-1}{2}}\log(\tau_j\sigma_k)(P_j^{\tau}\otimes P_k^{\sigma})(\mathcal{J}f_j)\right\|_{L_2(\Omega_n\times\Omega_n)}^{2}\cdot e^{-\sum_{m=1}^n\tau_m+\sigma_m}d\tau d\sigma.
\end{split}
\end{equation}
Since, for every $\tau,\sigma\in\mathbb{R}^n_+,$ the family $\{P^{\tau}_j\otimes P^{\sigma}_k\}_{j,k=1}^n$ consists of pairwise orthogonal projections, it follows that
\begin{equation}\label{Hilbert expansion 2}
\begin{split}
&\left\|\sum_{j,k=1}^n(\tau_j\sigma_k)^{\frac{z-1}{2}}\log(\tau_j\sigma_k)(P_j^{\tau}\otimes P_k^{\sigma})(\mathcal{J}f_j)\right\|_{L_2(\Omega_n\times\Omega_n)}^{2}\\
=&\sum_{j,k=1}^n(\tau_j\sigma_k)^{\Re(z)-1}\log^2(\tau_j\sigma_k)\|(P_j^{\tau}\otimes P_k^{\sigma})(\mathcal{J}f_j)\|_{L_2(\Omega_n\times\Omega_n)}^2\\
\leq &\sum_{j,k=1}^n(\tau_j\sigma_k)^{\Re(z)-1}\log^2(\tau_j\sigma_k)\|\mathcal{J}f_j\|_{L_2(\Omega_n\times\Omega_n)}^2\\
=&\sum_{j,k=1}^n(\tau_j\sigma_k)^{\Re(z)-1}\log^2(\tau_j\sigma_k)\|f_j\|_{L_2(\Omega_n)}^{2}.
\end{split}
\end{equation}
Hence, combining \eqref{Hilbert expansion 1} with \eqref{Hilbert expansion 2}, we have
\begin{align*}
&4\left\|\frac{d}{dz}\left(\mathcal{A}_{n,z}f\right)\right\|_{L_2({\rm Clock}_n;L_2(\Omega_{n}\times\Omega_{n}))}^2\\
\leq &\sum_{j,k=1}^n\|f_j\|_{L_2(\Omega_n)}^{2} \int_{\mathbb{R}^{n}_{+}\times \mathbb{R}^{n}_{+}}(\tau_j\sigma_k)^{\Re(z)-1}\log^2(\tau_j\sigma_k)e^{-\sum_{m=1}^n\tau_m+\sigma_m}d\tau d\sigma\\
=&\sum_{j,k=1}^n\|f_j\|_{L_2(\Omega_n)}^{2}\int_{\mathbb{R}^{n}_{+}\times \mathbb{R}^{n}_{+}}(uv)^{\Re(z)-1}\log^2(uv)e^{-u-v}dudv\\
=&n\sum_{j=1}^n\|f_j\|_{L_2(\Omega_n)}^{2}\times(\Gamma^2)''(\Re(z))<\infty.
\end{align*}
Hence, we obtain the analyticity of the family.
\end{proof}

The next lemma establishes the continuity of the mapping $z\mapsto M_{\1_{E_{\varepsilon}}}\mathbf{U}\mathcal{A}_{n,z}$ for all $z\in \mathbb{C}$, which is the key ingredient to apply the Stein interpolation theorem.
\begin{lem}\label{final continuity lemma} For every $h\in L_2(\Omega_n,\ell^{n}_{2}),$ the family
$$
 z\longmapsto M_{\1_{E_{\epsilon}}}\mathbf{U}\mathcal{A}_{n,z}h
$$
is $L_{\infty}({\rm Clock}_n;\BMO(\Omega_n\times\Omega_n))$-continuous on $\mathbb{C}.$ More precisely,
\begin{align*}
&\|M_{\1_{E_{\epsilon}}}\mathbf{U}\mathcal{A}_{n,z}h-M_{\mathbf{1}_{E_{\epsilon}}}\mathbf{U}\mathcal{A}_{n,w}h\|_{L_{\infty}({\rm Clock}_n;\BMO(\Omega_n\times\Omega_n))}\\
\leq & n\|{\rm id}\|_{L_2(\Omega_n\times\Omega_n)\to\BMO(\Omega_n\times\Omega_n)}\cdot \sup_{a\in[\epsilon^{-1},\epsilon]}|a^{z-1}-a^{w-1}|\cdot\|h\|_{L_2(\Omega_n,\ell^{n}_{2})}.
\end{align*}
\end{lem}
\begin{proof}
Clearly,
\begin{align*}
&\|M_{\1_{E_{\epsilon}}}\mathbf{U}\mathcal{A}_{n,z}h-M_{\1_{E_{\epsilon}}}\mathbf{U}\mathcal{A}_{n,w}h\|_{L_{\infty}({\rm Clock}_n;\BMO(\Omega_n\times\Omega_n))}\\
\leq & \|M_{\1_{E_{\epsilon}}}\mathbf{U}\mathcal{A}_{n,z}h-M_{\1_{E_{\epsilon}}}\mathbf{U}\mathcal{A}_{n,w}h\|_{L_{\infty}({\rm Clock}_n;L_2(\Omega_n\times\Omega_n))}\cdot\|{\rm id}\|_{L_2(\Omega_n\times\Omega_n)\to\BMO(\Omega_n\times\Omega_n)}.
\end{align*}
Using the equality $M_{\1_{E_{\epsilon}}}\mathbf{U}=\mathbf{U}M_{\1_{E_{\epsilon}}},$ we obtain
\begin{align*}
&\|M_{\1_{E_{\epsilon}}}\mathbf{U}\mathcal{A}_{n,z}h-M_{\1_{E_{\epsilon}}}\mathbf{U}\mathcal{A}_{n,w}h\|_{L_{\infty}({\rm Clock}_n;L_2(\Omega_n\times\Omega_n))}\\
=&\|\mathbf{U}\Big(M_{\1_{E_{\epsilon}}}\mathcal{A}_{n,z}h-M_{\1_{E_{\epsilon}}}\mathcal{A}_{n,w}h\Big)\|_{L_{\infty}({\rm Clock}_n;L_2(\Omega_n\times\Omega_n))}\\
=&\|M_{\1_{E_{\epsilon}}}\mathcal{A}_{n,z}h-M_{\1_{E_{\epsilon}}}\mathcal{A}_{n,w}h\|_{L_{\infty}({\rm Clock}_n;L_2(\Omega_n\times\Omega_n))}\\
=&\|\mathcal{A}_{n,z}h-\mathcal{A}_{n,w}h\|_{L_{\infty}(E_{\epsilon};L_2(\Omega_n\times\Omega_n))}\\
=&\sup_{(\tau,\sigma)\in E_{\epsilon}}\|(\mathcal{A}_{n,z}h-\mathcal{A}_{n,w}h)(\tau,\sigma)\|_{L_2(\Omega_n\times\Omega_n)}.
\end{align*}
For almost every $(\tau,\sigma)\in E_{\epsilon}$, by the pairwise orthogonality of projections $\{P^{\tau}_{j}\otimes P^{\sigma}_{k}\}_{1\leq j,k\leq n}$ we have
\begin{align*}
&\|(\mathcal{A}_{n,z}h-\mathcal{A}_{n,w}h)(\tau,\sigma)\|_{L_2(\Omega_n\times\Omega_n)}^2\\
=&\sum_{1\leq j,k\leq n}|\tau_j^{\frac{z-1}{2}}\sigma_k^{\frac{z-1}{2}}-\tau_j^{\frac{w-1}{2}}\sigma_k^{\frac{z-1}{2}}|^2\cdot \|(P_j^{\tau}\otimes P_k^{\sigma})(\mathcal{J}h_j)\|_{L_2(\Omega_n\times\Omega_n)}^{2}\\
\leq &\sum_{1\leq j,k\leq n}|\tau_j^{\frac{z-1}{2}}\sigma_k^{\frac{z-1}{2}}-\tau_j^{\frac{w-1}{2}}\sigma_k^{\frac{z-1}{2}}|^2\cdot \|\mathcal{J}h_j\|_{L_2(\Omega_n\times\Omega_n)}^2\\
\leq &\sum_{1\leq j,k\leq n}|\tau_j^{\frac{z-1}{2}}\sigma_k^{\frac{z-1}{2}}-\tau_j^{\frac{w-1}{2}}\sigma_k^{\frac{z-1}{2}}|^2\cdot \max_{1\leq j\leq n}\|\mathcal{J}h_j\|_{L_2(\Omega_n\times\Omega_n)}^{2}\\
\leq & n^2\sup_{a\in[\epsilon^{-1},\epsilon]}|a^{z-1}-a^{w-1}|^2\cdot\|h\|_{L_2(\Omega_n,\ell^{n}_{2})}^2.
\end{align*}
This completes our proof.
\end{proof}

The following lemma provides endpoints estimates of the operator $M_{\1_{E_{\varepsilon}}}\mathbf{U}\mathcal{A}_{n,z}$, which immediately follows from Section \ref{endpoint estimate section}.

\begin{lem}\label{final boundedness lemma} 
Keeping notations as above, for each $\epsilon\in (0,1)$, we have
$$\left\|M_{\1_{E_{\epsilon}}}\mathbf{U}\mathcal{A}_{n,z}\right\|_{L_2(\Omega_n,\ell^{n}_{2})\to L_2({\rm Clock}_n;L_2(\Omega_n\times\Omega_n))}\leq\pi^{\frac{1}{2}},\quad \Re(z)=\frac12,$$
and
$$\left\|M_{\1_{E_{\epsilon}}}\mathbf{U}\mathcal{A}_{n,z}\right\|_{L_{\infty}(\Omega_n,\ell^{n}_{2})\to L_{\infty}({\rm Clock}_n;\BMO(\Omega_n\times\Omega_n))}\leq 1,\quad \Re(z)=1.$$
\end{lem}
\begin{proof} By Theorems \ref{thm:left} and \ref{thm:right}, it follows that
$$\|\mathcal{A}_{n,z}\|_{L_2(\Omega_n,\ell^{n}_{2})\to L_2({\rm Clock}_n;L_2(\Omega_n\times\Omega_n))}\leq\pi^{\frac12},\quad \Re(z)=\frac{1}{2},$$
and
$$\|\mathbf{U}\mathcal{A}_{n,z}\|_{L_{\infty}(\Omega_n,\ell^{n}_{2})\to L_{\infty}({\rm Clock}_n;\BMO(\Omega_n\times\Omega_n))}\leq 1,\quad \Re(z)=1.$$
The following contractions are obvious
$$\|M_{\1_{E_{\epsilon}}}\mathbf{U}\|_{L_2({\rm Clock}_n;L_2(\Omega_n\times\Omega_n))\circlearrowleft}\leq 1$$
and
$$\|M_{\1_{E_{\epsilon}}}\|_{L_{\infty}({\rm Clock}_n;\BMO(\Omega_n\times\Omega_n))\circlearrowleft}\leq 1.$$
Combining all inequalities in the proof together yields the desired inequalities.
\end{proof}

We need one more interpolation lemma.
Set
$$Y_0=L_2({\rm Clock}_n;L_2(\Omega_n\times\Omega_n)),\quad Y_1=L_{\infty}({\rm Clock}_n;\BMO(\Omega_n\times\Omega_n)).$$

\begin{lem}\label{cor:clock-interpolation}
If $q\geq 2$ and $\theta=1-\frac{2}{q},$ then
\begin{equation}\label{eq:clock-interpolation-embedding}
\|\cdot\|_{L_q({\rm Clock}_n\times \Omega_n\times\Omega_n)}\leq c_{{\rm abs}}q^2\|\cdot\|_{[Y_0,Y_1]_{\theta}}.
\end{equation}
\end{lem}
\begin{proof} Theorem \ref{thm:bochner} gives
$$[Y_0,Y_1]_{\theta}=L_q({\rm Clock}_n;[L_2(\Omega_n\times\Omega_n),\BMO(\Omega_n\times\Omega_n)]_{\theta})$$
isometrically.  Apply Theorem \ref{thm:product-interpolation} pointwise and
integrate.
\end{proof}

We are now ready to prove the main result in this section.

\begin{proof}[Proof of Theorem \ref{key theorem}] Using Theorem \ref{standard complex interpolation theorem} (whose assumptions are verified in lemmas above), for the family of operators $z\to M_{\1_{E_{\epsilon}}}\mathbf{U}\mathcal{A}_{\frac{z+1}{2}},$ $0\leq\Re(z)\leq 1,$ we write
$$\|M_{\1_{E_{\epsilon}}}\mathbf{U}\mathcal{A}_{n,1-\frac1q}\|_{[L_2(\Omega_n,\ell^{n}_{2}),L_{\infty}(\Omega_n,\ell^{n}_{2})]_{\theta}\to [Y_0,Y_1]_{\theta}}
 \leq\pi^{\frac12},\quad \theta=1-\frac{2}{q}.$$
Recall that
$$[L_2(\Omega_n,\ell^{n}_{2}),L_{\infty}(\Omega_n,\ell^{n}_{2})]_{\theta}=L_q(\Omega_n;\ell^{n}_{2})$$
isometrically. Using this and Lemma \ref{cor:clock-interpolation}, we obtain
$$\|M_{\1_{E_{\epsilon}}}\mathbf{U}\mathcal{A}_{n,1-\frac1q}f\|_{L_q({\rm Clock}_n\times\Omega_n\times\Omega_n)}
 \leq c_{{\rm abs}}q^2\|f\|_{L_q(\Omega_n;\ell^{n}_{2})}.$$

The events $E_{\epsilon}$ increase to a full-measure subset of the clock space. Monotone convergence allows to pass $\epsilon\to 0,$ thus giving
$$\|\mathbf{U}\mathcal{A}_{n,1-\frac1q}f\|_{L_q({\rm Clock}_n\times\Omega_n\times\Omega_n)}
 \leq c_{{\rm abs}}q^2\|f\|_{L_q(\Omega_n;\ell^{n}_{2})}.$$
Observe that $\mathbf{U}$ is an isometry from $L_q({\rm Clock}_n\times\Omega_{n}\times\Omega_{n})$ onto itself. Thus,
$$\left\|\mathcal{A}_{n,1-\frac1q}f\right\|_{L_q({\rm Clock}_n\times\Omega_n\times\Omega_n)}
 \leq c_{{\rm abs}}q^2\|f\|_{L_q(\Omega_n;\ell^{n}_{2})}.$$
\end{proof}

\section{The sharp fractional Riesz estimate and its higher-order extension}\label{main result section}

\subsection{Proof of Theorem \ref{thm:main}} In this subsection we provide a detailed proof of the main result of the paper, Theorem \ref{thm:main}. We begin with the next lemma, which follows from the boundedness of $\mathcal{A}_{n,z}$ and probabilistic representation.
\begin{lem}\label{b boundedness lemma} For every $1<p\leq 2,$ we have
$$\left\|\mathcal{B}_{n,\frac{1}{p}}h\right\|_{L_{\frac{p}{p-1}}(\Omega_n)}\leq c_{{\rm abs}}\left(\frac{p}{p-1}\right)^2\|h\|_{L_{\frac{p}{p-1}}(\Omega_n;\ell^{n}_{2})},\quad h\in L_{\frac{p}{p-1}}(\Omega_n;\ell^{n}_{2}).$$
\end{lem}
\begin{proof}
Theorem \ref{key theorem} (applied with $q=\frac{p}{p-1}$) asserts
$$\|\mathcal{A}_{n,\frac1p}h\|_{L_\frac{p}{p-1}({\rm Clock}_n\times\Omega_n\times\Omega_n)}\leq c_{{\rm abs}}\left(\frac{p}{p-1}\right)^2\|h\|_{L_{\frac{p}{p-1}}(\Omega_n;\ell^{n}_{2})}.$$
Since the conditional expectation
$$\mathbb{E}_{\nu^{\otimes n}\otimes \nu^{\otimes n}}:L_\frac{p}{p-1}({\rm Clock}_n\times\Omega_n\times\Omega_n)\to L_{\frac{p}{p-1}}(\Omega_n\times\Omega_n)$$
is a contraction, it follows that
$$
\left\|\left(\mathbb{E}_{\nu^{\otimes n}\otimes \nu^{\otimes}}\circ\mathcal{A}_{n,\frac1p}\right)h\right\|_{L_{\frac{p}{p-1}}(\Omega_n\times\Omega_n)}\leq c_{{\rm abs}}\left(\frac{p}{p-1}\right)^2\|h\|_{L_{\frac{p}{p-1}}(\Omega_n;\ell^{n}_{2})}.
$$
By Proposition \ref{thm:averaging},
$$\left(\mathbb{E}_{\nu^{\otimes n}\otimes \nu^{\otimes n}}\circ\mathcal{A}_{n,\frac1p}\right)h=\Gamma^2\left(\frac{p+1}{2p}\right)\mathcal{J}\left(\mathcal{B}_{n,\frac1p}h\right),$$
we have
$$
\left\|\mathcal{J}\left(\mathcal{B}_{n,\frac1p}h\right)\right\|_{L_{\frac{p}{p-1}}(\Omega_n\times\Omega_n)}\leq c_{{\rm abs}}(\frac{p}{p-1})^2\big(\Gamma(\frac{p+1}{2p})\big)^{-2}\|h\|_{L_{\frac{p}{p-1}}(\Omega_n;\ell^{n}_{2})}.
$$
Since the mapping $\mathcal{J}:L_{\frac{p}{p-1}}(\Omega_n)\to L_{\frac{p}{p-1}}(\Omega_n\times\Omega_n)$ is an isometry, the assertion follows.
\end{proof}

We now prove Theorem \ref{thm:main}. 

\begin{proof}[Proof of Theorem \ref{thm:main}] According to Lemma \ref{lem:dual-reduction}, $\nabla\Delta^{-\frac1p}:L_p(\Omega_n)\to L_p(\Omega_n,\ell^{n}_{2})$ is exactly $\mathcal{B}_{n,\frac1p}:L_{\frac{p}{p-1}}(\Omega_n;\ell^{n}_{2})\to L_{\frac{p}{p-1}}(\Omega_n).$ It follows now from Lemma \ref{b boundedness lemma} that
$$
\|\nabla\Delta^{-\frac1p}\|_{L_p(\Omega_n)\to L_p(\Omega_n,\ell^{n}_{2})}\leq c_{{\rm abs}}\left(\frac{p}{p-1}\right)^{2}.
$$
This completes the proof.
\end{proof}

\subsection{Proof of Theorem \ref{thm:higher-main}} In this subsection  we prove Theorem \ref{thm:higher-main} which is the higher-order extension of Theorem \ref{thm:main}. Our argument is based on Hilbert-valued amplification, with the constant  independent of both the dimension of the hypercube and the auxiliary Hilbert dimension.

The following Hilbert-valued form of the Kahane--Khintchine inequality is well-known (see, for example,
\cite[Chapter~4]{LedouxTalagrand}). For $1\leq p\leq2,$ for every sequence $(v_l)_{l\geq0}$ in a Hilbert space $H,$
\begin{equation}\label{eq:higher-khintchine}
\left(\sum_{l\geq0}\|v_l\|_H^2\right)^{\frac{1}{2}}\leq \sqrt{2}\left(\int_0^1\left\|\sum_{l\geq0}r_l(t)v_l\right\|_H^pdt\right)^{\frac1p}.
\end{equation}

\begin{lem}[Hilbert-valued amplification]\label{lem:hilbert-amplification}

Let $(X,\mu)$ be a probability space. Let $\{T_{j}\}_{j\geq0}$ be a family of linear maps from $L_p(X)$ into itself with $1\leq p\leq 2$ such that
$$
\left\|\left(\sum_{j\geq0}|T_jf|^2\right)^{\frac12}\right\|_{L_p(X)}\leq\|f\|_{L_p(X)},\quad f\in L_p(X).
$$
For every separable Hilbert space $H$ we have
$$
\left\|\left(\sum_{j\geq0}\|T_jF\|_H^2\right)^{\frac12}\right\|_{L_p(X)}\leq \sqrt{2}\|F\|_{L_p(X;H)},\quad f\in L_p(X;H).
$$
\end{lem}
\begin{proof} Choose an orthonormal basis $(e_l)_{l\geq0}$ of $H$ and write
$F=\sum_{l\geq0}f_le_l.$ 
For each $j$, we have
$$\|T_jF\|_H^2=\sum_{l\geq0}|T_jf_l|^2.$$
Thus,
$$
\left\|\left(\sum_{j\geq0}\|T_jF\|_H^2\right)^{1/2}\right\|_{L_p(X)}=\left\|\left(\sum_{j\geq0}\sum_{l\geq0}|T_jf_l|^2\right)^{\frac12}\right\|_{L_p(X)}.
$$
Set $v_l=\{T_jf_l\}_{j\geq0}\in \ell_{2}$ for each $l\in \mathbb{N}$. The following identity holds trivially:
$$
\left\|\left(\sum_{j\geq0}\|T_jF\|_H^2\right)^{1/2}\right\|_{L_p(X)}=\left\|\left(\sum_{l\geq0}\|v_l\|_{l_2}^2\right)^{\frac12}\right\|_{L_p(X)}.
$$
Using \eqref{eq:higher-khintchine}, we write
\begin{align*}
\left\|\left(\sum_{j\geq0}\|T_jF\|_H^2\right)^{1/2}\right\|_{L_p(X)}\leq&\sqrt{2}\left\|\left(\int_0^1\left\|\sum_{l\geq0}r_l(t)v_l\right\|_{\ell_{2}}^pdt\right)^{\frac1p}\right\|_{L_p(X)}\\
=&\sqrt{2}\left(\int_0^1\left\|\left\|\sum_{l\geq0}r_l(t)v_l\right\|_{\ell_{2}}\right\|_{L_p(X)}^pdt\right)^{\frac1p}.
\end{align*}
The following holds true $\mu$-almost everywhere
$$
\left\|\sum_{l\geq0}r_l(t)v_l\right\|_{\ell_{2}}=\left\|\left\{T_j\left(\sum_{l\geq0}r_l(t)f_l\right)\right\}_{j\geq0}\right\|_{\ell_{2}},\
$$
and, therefore, implies
$$
\left\|\left\|\sum_{l\geq0}r_l(t)v_l\right\|_{\ell_{2}}\right\|_{L_p(X)}=\left\|\left\{T_j\left(\sum_{l\geq0}r_l(t)f_l\right)\right\}_{j\geq0}\right\|_{L_p(X,\ell_2)}\leq\left\|\sum_{l\geq0}r_l(t)f_l\right\|_{L_p(X)}.
$$
Thus, combining estimates provided as above, we have
\begin{equation}\label{Hilbert calculus 1}
\begin{split}
\left\|\left(\sum_{j\geq0}\|T_jF\|_H^2\right)^{1/2}\right\|_{L_p(X)}&\leq \sqrt{2}\left(\int_0^1\left\|\sum_{l\geq0}r_l(t)f_l\right\|_{L_p(X)}^pdt\right)^{\frac{1}{p}}\\
&=\sqrt{2}\left\|\sum_{l\geq0}r_l\otimes f_l\right\|_{L_p((0,1)\times X)}.
\end{split}
\end{equation}

Applying Khinchine inequality to \eqref{Hilbert calculus 1} we get
\begin{align*}
\left\|\left(\sum_{j\geq0}\|T_jF\|_H^2\right)^{1/2}\right\|_{L_p(X)}&\leq \sqrt{2}\left\|\sum_{l\geq0}r_l\otimes f_l\right\|_{L_p((0,1)\times X)}\\
&\leq \sqrt{2}\left\|\left(\sum_{l\geq0}|f_l|^2\right)^{\frac{1}{2}}\right\|_{L_p(X)}=\sqrt{2}\|F\|_{L_{p}(X;H)}.
\end{align*}
\end{proof}

We now combine Lemma \ref{lem:hilbert-amplification} with Theorem \ref{thm:main} to give a proof of Theorem \ref{thm:higher-main}.
\begin{proof}[Proof of Theorem \ref{thm:higher-main}]
Fix $k\in\mathbb{N}$ and denote $F=\nabla^{k-1}\Delta^{-\frac{k-1}{p}}f.$ We have
$$\nabla^k\Delta^{-\frac{k}{p}}f=(\nabla\Delta^{-\frac1p})(F).$$
Thus,
$$
\left\|\nabla^k\Delta^{-\frac{k}{p}}f\right\|_{L_p(\Omega_n,l_2^{n^k})}=\left\|\left(\sum_{j\geq0}\|D_j\Delta^{-\frac1p}F\|_H^2\right)^{\frac12}\right\|_{L_p(\Omega_n)}.
$$
Here, $H=l_2^{n^{k-1}}.$ Using Lemma \ref{lem:hilbert-amplification} and Theorem \ref{thm:main}, we write
\begin{align*}
\left\|\nabla^k\Delta^{-\frac{k}{p}}f\right\|_{L_p(\Omega_n;\ell_2^{n^k})}&\leq \sqrt{2}\cdot c_{{\rm abs}}(p-1)^{-2}\cdot \|F\|_{L_p(\Omega_n;\ell_2^{n^{k-1}})}\\
&=\sqrt{2}\cdot c_{{\rm abs}}(p-1)^{-2}\cdot\|\nabla^{k-1}\Delta^{-\frac{k-1}{p}}g\|_{L_p\left(\Omega_n;\ell_2^{n^{k-1}}\right)}.
\end{align*}
Here, $c_{{\rm abs}}$ is the constant in Theorem \ref{thm:main}. The assertion follows now by induction.
\end{proof}

\subsection{Exponent $\frac{k}{p}$ is optimal}\label{optimality on exponent}

Now, we prove that the power $\frac{k}{p}$ cannot be lowered, even in the distinct-index estimate. This follows by combining the next two lemmas.
\begin{lem} If $n\geq 2k$ and $F_n=\chi_{\mathbf{1}},$  then
$$\Big\|\Big(
 \sum_{\substack{\mathbf j\in [n]^k\\
 j_1<j_2<\cdots<j_k}}
 |D_{\mathbf j}F_n|^{2}
\Big)^{\frac12}
\Big\|_{L_p(\Omega_n)}
 \geq \frac{n^{\frac{k}{p}}}{k!}\|F_n\|_{L_p(\Omega_n)}.$$
\end{lem}
\begin{proof} We have
$$F_n(x)=2^{-n}\prod_{j=1}^n(1+x_j),\quad x\in\Omega_n.$$
In other words,
$$F_n=2^{-n}\sum_{A\subset [n]}w_A.$$
Using the identity
$$D_{\mathbf{j}}w_A=\prod_{m=1}^k\1_A(j_m) w_A,$$
we write
$$D_{\mathbf{j}}F_n=2^{-n}\sum_{\substack{A\subset [n]\\ j_1,\cdots,j_k\in A}}w_A.$$

\noindent For a given $\mathbf{j}$ with $j_1<j_2<\cdots<j_k,$ let $x_{\mathbf{j}}(j_m)=-1$ for $1\leq m\leq k$ and $x_{\mathbf{j}}(j)=1$ if $j\neq j_1,\cdots,j_k.$ Hence,
$$(D_{\mathbf{j}}F_n)(x_{\mathbf{j}})=2^{-n}\sum_{\substack{A\subset [n]\\ j_1,\cdots,j_k\in A}}w_A(x_{\mathbf{j}})=2^{-n}\sum_{\substack{A\subset [n]\\ j_1,\cdots,j_k\in A}}(-1)^k=(-2)^k.$$
In other words,
$$|D_{\mathbf{j}}F_n|\geq 2^k\1_{\{x_{\mathbf{j}}\}}.$$
Indeed, it is obvious that
$$\Big\|\Big(
 \sum_{\substack{\mathbf j\in [n]^{k}\\
 j_1<j_2<\cdots<j_k}}
 |D_{\mathbf j}F_n|^{2}
\Big)^{\frac12}
\Big\|_{L_p(\Omega_n)}\geq 2^k\Big\|\Big(
 \sum_{\substack{\mathbf j\in [n]^{k}\\
 j_1<j_2<\cdots<j_k}}\1_{\{x_{\mathbf{j}}\}}\Big)^{\frac12}\Big\|_{L_p(\Omega_n)}=2^{k-\frac{n}{p}}\binom{n}{k}^{\frac1p}.$$
It now remains to note that $\|F_{n}\|_{L_p(\Omega_n)}=2^{-\frac{n}{p}}$ and that
$$\binom{n}{k}\geq \frac{n^k}{2^k\cdot k!},\quad n\geq 2k.$$
\end{proof}

\begin{lem} We have
$$\|\Delta^{\gamma}\|_{L_p(\Omega_n)\to L_p(\Omega_n)}\leq 3n^{\gamma},\quad \gamma>0.$$
\end{lem}
\begin{proof}  Each $E_j,$ $1\leq j\leq n,$ is a contraction on $L_p(\Omega_n).$ Hence, so is $D_j=M_{x_j}\circ E_j\circ M_{x_j}.$  Clearly,
$$\|\Delta\|_{L_p(\Omega_n)\circlearrowleft}\leq\sum_{j=1}^n\|D_j\|_{L_p(\Omega_n)\circlearrowleft}\leq\sum_{j=1}^n1=n.$$

The semigroup $(e^{-t\Delta})_{t>0}$ is contractive. We obviously have
$$\|1-e^{-t\Delta}\|_{L_p(\Omega_n)\to L_p(\Omega_n)}\leq\|1\|_{L_p(\Omega_n)\circlearrowleft}+\|e^{-t\Delta}\|_{L_p(\Omega_n)\circlearrowleft}=1.$$
On the other hand, we have
$$\|1-e^{-t\Delta}\|_{L_p(\Omega_n)\circlearrowleft}=\|\sum_{k\geq1}\frac{(-t\Delta)^k}{k!}\|_{L_p(\Omega_n)\circlearrowleft}\leq\sum_{k\geq1}\frac{t^k}{k!}\|\Delta^k\|_{L_p(\Omega_n)\circlearrowleft}=$$
$$=\sum_{k\geq1}\frac{t^k}{k!}\|\Delta\|_{L_p(\Omega_n)\circlearrowleft}^k\leq\sum_{k\geq1}\frac{(tn)^k}{k!}=e^{tn}-1.$$
For $0<\gamma<1,$ we write
$$\Delta^{\gamma}=c_{\gamma}\int_0^{\infty}(1-e^{-t\Delta})\frac{dt}{t^{1+\gamma}},\mbox{ where }c_{\gamma}^{-1}=\int_0^{\infty}(1-e^{-t})\frac{dt}{t^{1+\gamma}}.$$
Hence,
$$\|\Delta^{\gamma}\|_{L_p(\Omega_n)\to L_p(\Omega_n)}\leq c_{\gamma}\int_0^{\infty}\min\{2,e^{tn}-1\}\frac{dt}{t^{1+\gamma}}=c_{\gamma}n^{\gamma}\int_0^{\infty}\min\{2,e^t-1\}\frac{dt}{t^{1+\gamma}}.$$
Since
$$\frac{\min\{2,e^t-1\}}{1-e^{-t}}\leq 3,\quad t>0,$$
it follows that
$$\|\Delta^{\gamma}\|_{L_p(\Omega_n)\to L_p(\Omega_n)}\leq 3n^{\gamma},\quad 0<\gamma<1.$$
This yields the assertion for $0<\gamma<1$. For a general $\gamma,$ we write
$$\|\Delta^{\gamma}\|_{L_p(\Omega_n)\to L_p(\Omega_n)}\leq \|\Delta^{\gamma-\lfloor\gamma\rfloor}\|_{L_p(\Omega_n)\to L_p(\Omega_n)}\|\Delta\|_{L_p(\Omega_n)\to L_p(\Omega_n)}^{\lfloor\gamma\rfloor}\leq 3n^{\gamma-\lfloor\gamma\rfloor}\cdot n^{\lfloor \gamma\rfloor}.$$
This completes the proof.
\end{proof}

\section{Lower estimates for the higher-order constants}\label{optimality on constants}

We do not know whether constant $(p-1)^{-2k}$ obtained in Theorems \ref{thm:main} and \ref{thm:higher-main} is optimal for $p$ close to $1.$ As a partial relief, we provide the following norm estimate from below. 

\begin{prop}\label{prop:constants} For every $k\geq1,$ we have
$$\sup_{n\geq1}\|\nabla^k\Delta^{-\frac{k}{p}}\|_{L_p(\Omega_n)\to L_p(\Omega_n;\ell^{n^k}_2)}\geq c_k(p-1)^{-\frac{k}{2}},\quad 1<p\leq 2.$$
\end{prop}

\begin{lem}\label{rearrangement lemma} For every finitely supported mapping $f:\mathbb{Z}^k_+\to\mathbb{C},$ we have 
$$\sum_{\substack{\mathbf{j}\in\mathbb{Z}_+^k\\
j_1,\cdots,j_k\mbox{ are distinct}}}f(\mathbf{j})=\sum_{\sigma\in\Pi_k}
\Big(\prod_{B\in\sigma}
(-1)^{|B|-1}(|B|-1)!\Big)\sum_{\substack{\mathbf{j}\in\mathbb{Z}_+^k\\m\sim_\sigma l\Rightarrow j_m=j_l}}
f(j_1,\ldots,j_k).$$
Here, $\Pi_k$ denote the lattice of partitions of $[k].$
\end{lem}
\begin{proof} For a tuple $\mathbf{j}=(j_1,\cdots,j_k),$ let $\ker(\mathbf{j})\in\Pi_k$ denote its kernel partition:
$$m\sim_{\ker(\mathbf{j})}l \mbox{ iff }j_m=j_l.$$
Thus, $j_1,\cdots,j_k$ are all distinct if and only if $\ker(\mathbf{j})=\hat{0},$ where $\hat{0}$ is the discrete partition. For each $\pi\in\Pi_k,$ define
$$F(\pi)=\sum_{\substack{\mathbf{j}\in\mathbb{Z}_+^k\\
m\sim_\pi l\Rightarrow j_m=j_l}}f(\mathbf{j}).$$
The condition in the inner sum is equivalent to $\ker(\mathbf{j})\geq\pi$ in the partition lattice. Hence, if we define
$$G(\sigma)=\sum_{\substack{\mathbf{j}\in\mathbb{Z}_+^k\\ \ker(j)=\sigma}}f(\mathbf{j}),$$
then
$$F(\pi)=\sum_{\sigma\geq\pi}G(\sigma).$$

Using M\"obius inversion formula (see \cite[Proposition 3.7.2]{Stan2012}), we write
$$G(\pi)=\sum_{\sigma\geq\pi}\mu(\pi,\sigma)F(\sigma).$$
Taking $\pi=\hat{0},$ we obtain
$$G(\hat{0})=\sum_{\sigma\in\Pi_k}\mu(\hat{0},\sigma)F(\sigma).$$
Since
$$G(\hat{0})=\sum_{\substack{\mathbf{j}\in\mathbb{Z}_+^k\\
\ker(\mathbf{j})=\hat{0}}}f(\mathbf{j})=\sum_{\substack{\mathbf{j}\in\mathbb{Z}_+^k\\
j_1,\cdots,j_k\mbox{ are distinct}}}f(\mathbf{j}),$$
it follows that
$$\sum_{\substack{\mathbf{j}\in\mathbb{Z}_+^k\\
j_1,\cdots,j_k\mbox{ are distinct}}}f(\mathbf{j})=\sum_{\sigma\in\Pi_k}
\mu(\hat 0,\sigma)\sum_{\substack{\mathbf{j}\in\mathbb{Z}_+^k\\m\sim_\sigma l\Rightarrow j_m=j_l}}
f(j_1,\ldots,j_k).$$

Finally, the M\"obius function of the partition lattice satisfies (see \cite[Example 3.10.4, pp. 318--319]{Stan2012})
$$
\mu(\hat{0},\sigma)=\prod_{B\in\sigma}(-1)^{|B|-1}(|B|-1)!.
$$
Combining the last two equalities, we complete the proof.
\end{proof}

\begin{lem}\label{convergence in distribution lemma} We have
$$n^{-\frac{k}{2}}\sum_{\substack{\mathbf{j}\in [n]^k\\ j_1,\cdots,j_k\mbox{ are distinct}}}r_{j_1}\cdots r_{j_k}\to H_k(G),\quad n\to\infty,$$
in distribution. Here, $H_k$ is the probabilist's Hermite polynomial and $G$ is the standard Gaussian random variable.
\end{lem}
\begin{proof} Using Lemma \ref{rearrangement lemma}, we write
$$\sum_{\substack{\mathbf{j}\in [n]^k\\
j_1,\cdots,j_k\mbox{ are distinct}}}r_{j_1}\cdots r_{j_k}=\sum_{\sigma\in\Pi_k}
\Big(\prod_{B\in\sigma}
(-1)^{|B|-1}(|B|-1)!\Big)\sum_{\substack{\mathbf{j}\in [n]^k\\m\sim_\sigma l\Rightarrow j_m=j_l}}
r_{j_1}\cdots r_{j_k}.$$
Clearly,
$$\sum_{\substack{\mathbf{j}\in [n]^k\\m\sim_\sigma l\Rightarrow j_m=j_l}}
r_{j_1}\cdots r_{j_k}=\prod_{B\in\sigma}(\sum_{j=1}^nr_j^{|B|})=(\sum_{j=1}^n)^{a_{\sigma}}\cdot n^{b_{\sigma}},$$
where $a_{\sigma}$ is the number of $B\in\sigma$ with $|B|$ being odd number, while $b_{\sigma}$ is the number of $B\in\sigma$ with $|B|$ being even number.
Thus,
$$\sum_{\substack{\mathbf{j}\in [n]^k\\
j_1,\cdots,j_k\mbox{ are distinct}}}r_{j_1}\cdots r_{j_k}=\sum_{\sigma\in\Pi_k}
\Big(\prod_{B\in\sigma}
(-1)^{|B|-1}(|B|-1)!\Big)\cdot (\sum_{j=1}^nr_j)^{a_{\sigma}}\cdot n^{b_{\sigma}}.$$
Dividing by $n^{\frac{k}{2}},$ we obtain
\begin{align*}
&n^{-\frac{k}{2}}\sum_{\substack{\mathbf{j}\in [n]^k\\
j_1,\cdots,j_k\mbox{ are distinct}}}r_{j_1}\cdots r_{j_k}\\
=&\sum_{\sigma\in\Pi_k}\left(\prod_{B\in\sigma}
(-1)^{|B|-1}(|B|-1)!\right)\cdot \left(n^{-\frac12}\sum_{j=1}^nr_j\right)^{a_{\sigma}}\cdot n^{b_{\sigma}+\frac12a_{\sigma}-\frac{k}{2}}.
\end{align*}

Consider polynomials
$$P_{k,n}:t\to \sum_{\sigma\in\Pi_k}
\Big(\prod_{B\in\sigma}
(-1)^{|B|-1}(|B|-1)!\Big)\cdot t^{a_{\sigma}}\cdot n^{b_{\sigma}+\frac12a_{\sigma}-\frac{k}{2}},\quad t\in\mathbb{R}.$$
We have
$$n^{-\frac{k}{2}}\sum_{\substack{\mathbf{j}\in [n]^k\\
j_1,\cdots,j_k\mbox{ are distinct}}}r_{j_1}\cdots r_{j_k}=P_{k,n}(n^{-\frac12}\sum_{j=1}^nr_j).$$

Consider polynomial
$$P_k:t\to \sum_{\substack{\sigma\in\Pi_k\\ \mbox{ if }B\in\sigma\mbox{ then }|B|=1\mbox{ or }|B|=2}}
\Big(\prod_{B\in\sigma}
(-1)^{|B|-1}(|B|-1)!\Big)\cdot t^{a_{\sigma}},\quad t\in\mathbb{R}.$$
Obviously, coefficients of $P_{k,n}$ converge to that of $P_k$ as $n\to\infty.$ Using Central Limit Theorem, we conclude that
$$n^{-\frac{k}{2}}\sum_{\substack{\mathbf{j}\in [n]^k\\
j_1,\cdots,j_k\mbox{ are distinct}}}r_{j_1}\cdots r_{j_k}\to P_k(G),\quad n\to\infty,$$
in distribution. Obviously,
$$P_k(t)=\sum_{\substack{\sigma\in\Pi_k\\ \mbox{ if }B\in\sigma\mbox{ then }|B|=1\mbox{ or }|B|=2}}
(-1)^{b_{\sigma}}\cdot t^{a_{\sigma}}=\sum_{a+2b=k}\binom{k}{a} (2b-1)!!(-1)^bt^a,$$
which exactly coincides with probabilist's Hermite polynomial $H_k.$
\end{proof}

\begin{lem}\label{hermite gauss lemma} We have
$$\liminf_{n\to\infty}n^{-\frac{k}{2}}\|\Delta^{-\frac{k}{p}}\sum_{\substack{\mathbf{j}\in [n]^k\\ j_1<\cdots<j_k}}D_{\mathbf{j}}w_{\{j_1,\cdots,j_k\}}\|_{L_q(\Omega_n)}\geq \frac1{k!\cdot k^{\frac{k}{p}}}\|H_k(G)\|_q.$$
Here, $H_k$ is the probabilist's Hermite polynomial and $G$ is the standard Gaussian random variable.
\end{lem}
\begin{proof}
Using the identity
$$
D_{\mathbf{j}}w_A=\prod_{m=1}^k\1_A(j_m) w_A,
$$
we write
$$
D_{\mathbf{j}}w_{\{j_1,\cdots,j_k\}}=w_{\{j_1,\cdots,j_k\}},\quad \Delta w_{\{j_1,\cdots,j_k\}}= k w_{\{j_1,\cdots,j_k\}}.
$$
Thus,
\begin{align*}
&\Delta^{-\frac{k}{p}}\sum_{\substack{\mathbf{j}\in [n]^k\\ j_1<\cdots<j_k}}D_{\mathbf{j}}w_{\{j_1,\cdots,j_k\}}=k^{-\frac{k}{p}}\sum_{\substack{\mathbf{j}\in [n]^k\\ j_1<\cdots<j_k}}w_{\{j_1,\cdots,j_k\}}\\
=&\frac1{k!\cdot k^{\frac{k}{p}}}\sum_{\substack{\mathbf{j}\in [n]^k\\ j_1,\cdots,j_k\mbox{ are distinct}}}r_{j_1}\cdots r_{j_k}.
\end{align*}
The assertion follows from Lemma \ref{convergence in distribution lemma} and the Fatou lemma.
\end{proof}

\begin{proof}[Proof of Proposition \ref{prop:constants}] The adjoint of
$\nabla^k\Delta^{-\frac{k}{p}}$ is given by the formula
$$h\to \Delta^{-\frac{k}{p}}\sum_{\mathbf{j}\in [n]^{k}}D_{\mathbf{j}}h_{\mathbf{j}}.$$
Set
$$h_{\mathbf{j}}=
\begin{cases}
w_{\{j_1,\cdots,j_k\}},&j_1<\cdots<j_k,\\
0,&\mbox{ otherwise}
 \end{cases}$$
Obviously, $\sum_{\mathbf{j}\in [n]^k}|h_{\mathbf{j}}|^2=\binom{n}{k}.$ By Lemma \ref{hermite gauss lemma},
$$\liminf_{n\to\infty}n^{-\frac{k}{2}}\|\Delta^{-\frac{k}{p}}\sum_{\mathbf{j}\in [n]^{k}}D_{\mathbf{j}}h_{\mathbf{j}}\|_{L_q(\Omega_n)}\geq \frac1{k!\cdot k^{\frac{k}{p}}}\|H_k(G)\|_q.$$
Noting that $\|H_k(G)\|_q\geq c_kq^{\frac{k}{2}},$ we complete the proof.
\end{proof}


\section{Applications of the sharp fractional Riesz estimates}\label{sec:applications}

The endpoint nature of the fractional Riesz estimate proved in Theorem \ref{thm:main} becomes particularly useful when it is combined with spectral information for the number operator $\Delta$.  The basic idea is simple but useful: analytic or polynomial estimates are first obtained for a fractional power of $\Delta$, and the sharp fractional Riesz estimate (i.e., Theorem \ref{thm:main}) then converts them into first-order gradient estimates without any loss in the spectral exponent.  This principle yields  the optimal short-time gradient smoothing of the heat semigroup \cite{E-I2020} and a logarithm-free Bernstein--Markov inequality for low-degree Walsh polynomials \cite{V2024}.  

We begin with the following dimension-free estimate, which was obtained by Eskenazis and Ivanisvili \cite{E-I2020} in their analysis of polynomial inequalities on the hypercube. From the present perspective, its short-time component is an immediate consequence of Theorem \ref{thm:main}. More importantly, the optimality argument below shows that the power $t^{-1/p}$ is not merely consistent with the sharp fractional Riesz estimate, but is in fact forced by the same endpoint threshold.

\begin{thm}[Eskenazis-Ivanisvili]\label{thm:heat-smoothing}
Let $1<p<2.$ There exists $A_p,B_p>0,$ such that 
\begin{equation}\label{eq:heat-smoothing}
\|\nabla e^{-t\Delta}\|_{L_p(\Omega_{n})\to L_p(\Omega_{n};\ell_2)}
\leq
\begin{cases}
A_pt^{-\frac1p},&0<t\leq1,\\
B_pe^{-t},&t\geq1.
\end{cases}
\end{equation}
The exponent $\frac1p$ is optimal in the following sense: if there exists constant $C_p$ such that
$$\|\nabla e^{-t\Delta}\|_{L_p(\Omega_{n})\to L_p(\Omega_{n};\ell_2)}\leq C_pt^{-\gamma},\quad t\in(0,1],$$
then $\gamma\geq\frac1p.$
\end{thm}

Before proving Theorem \ref{thm:heat-smoothing}, we recall the standard analytic-semigroup estimate that supplies the fractional smoothing of $e^{-t\Delta}$ (see \cite{CarbonaroDragicevic,Cowling1983}). For $1<p<\infty$ and $1\geq \alpha>0$,
\begin{equation}\label{eq:semigroup-functional-calculus}
\|\Delta^{\alpha}e^{-t\Delta}\|_{L_p(\Omega_{n})\circlearrowleft}
 \leq M_{p,\alpha}t^{-\alpha},\quad t>0,
\end{equation}
Indeed, since $\{e^{-t\Delta}\}_{t\geq 0}$ is a symmetric diffusion semigroup on $\Omega_{n}$. Hence, by Stein's analyticity theorem for symmetric diffusion semigroups \cite[Chapter~III, \S2, Theorem 1, p. 67]{Stein1970}, it extends to an analytic contraction semigroup on $L_{p}(\Omega_{n})$ for every $1<p<\infty$. By the Cauchy formula, it follows that for any $f$, $g\in L_{p}(\Omega_n)$ with $\|f\|_{L_{p}(\Omega_{n})}=\|g\|_{L_{\frac{p}{p-1}}(\Omega_{n})}=1$, then we have
\begin{align*}
\left|\phi_{g,f}^{\prime}(t)\right|&=\left|\frac{1}{2\pi i}\int_{\gamma(t)}\frac{\phi_{g,f}(z)}{(z-t)^{2}}~dz\right|\\
&\leq \frac{1}{2\pi}\int_{0}^{2\pi}\frac{\sup_{z\in \Sigma_{\theta}}\|e^{-z\Delta}\|_{L_{p}(\Omega_{n})\circlearrowleft}}{t|\sin(\theta)|}~d\theta\\
&=\frac{\sup_{z\in \Sigma_{\theta}}\|e^{-z\Delta}\|_{L_{p}(\Omega_{n})\circlearrowleft}}{t},
\end{align*}
where $\phi_{g,f}(z)=\langle g, e^{-z\Delta}f\rangle$ and $\gamma(t)$ stands for the circle of radius $t\sin(\theta)$ with center $t>0$ for some $\theta_{p}\in (0,\frac{\pi}{2})$. This verifies that for all $t>0$
\begin{equation}\label{usual power}
\left\|\Delta e^{-t\Delta}\right\|_{L_{p}(\Omega_{n})\circlearrowleft}\leq \frac{\sup_{z\in \Sigma_{\theta}}\|e^{-z\Delta}\|_{L_{p}(\Omega_{n})\circlearrowleft}}{t}.
\end{equation}
Thus, applying \eqref{usual power} we have
\begin{align*}
\left\|\Delta^{\alpha}e^{-t\Delta}(f)\right\|_{L_{p}(\Omega_{n})\circlearrowleft}&\leq \left\|e^{-t\Delta}(f)\right\|^{1-\alpha}_{L_{p}(\Omega_{n})\circlearrowleft}\left\|\Delta e^{-t\Delta}(f)\right\|^{\alpha}_{L_{p}(\Omega_{n})\circlearrowleft}\\
&\leq \sup_{z\in \Sigma_{\theta}}\|e^{-z\Delta}\|_{L_{p}(\Omega_{n})\circlearrowleft}^{\alpha}t^{-\alpha}\|f\|_{L_{p}(\Omega_{n})\circlearrowleft}.
\end{align*}

\begin{lem}\label{exponential decay lemma} For every mean zero function $g,$ we have
$$\|e^{-t\Delta}g\|_{L_p(\Omega_{n})}\leq (p-1)^{-\frac12}e^{-t}\|g\|_{L_p(\Omega_{n})},\quad t>0,$$
and
$$\|\nabla e^{-t\Delta}g\|_{L_p(\Omega_{n};\ell_2)}\leq c_{{\rm abs}}(p-1)^{-\frac52}\|\Delta^{\frac1p}e^{-\Delta}\|_{L_p(\Omega_{n})\circlearrowleft}\cdot e^{1-t}\|g\|_{L_p(\Omega_{n})},\quad t\geq 1.$$
\end{lem}
\begin{proof} Let
$s_p=\frac12\log(\frac{1}{p-1}).$  By the hypercontractivity of the semigroup (see for instance \cite{Beckner1975,Bonami1970,Ne1973}) we obtain
$$
\|e^{-s_p\Delta}\|_{L_p(\Omega_{n})\to L_2(\Omega_{n})}\leq1.
$$
Clearly, ${\rm spec}(\Delta)\subset\mathbb{Z}_+$ and the eigenspace corresponding to the eigenvalue $0$ consists of constant functions. If $\mathbb{E} g=0$ and $t\geq s_p,$ then spectral gap yields
$$\|e^{-t\Delta}g\|_{L_p(\Omega_{n})}\leq\|e^{-t\Delta}g\|_{L_2(\Omega_{n})}\leq e^{-(t-s_p)}\|e^{-s_p\Delta}g\|_{L_2(\Omega_{n})}\leq e^{-(t-s_p)}\|g\|_{L_p(\Omega_{n})}.$$
On the other hand, if $t\leq s_p,$ then
$$\|e^{-t\Delta}g\|_{L_p(\Omega_{n})}\leq\|g\|_{L_p(\Omega_{n})}\leq e^{-(t-s_p)}\|g\|_{L_p(\Omega_{n})}.$$
Thus, for every $t\geq0,$
$$\|e^{-t\Delta}g\|_{L_p(\Omega_{n})}\leq e^{-(t-s_p)}\|g\|_{L_p(\Omega_{n})}=(p-1)^{-\frac12}e^{-t}\|g\|_{L_p(\Omega_{n})}.$$
This yields the first inequality.

Hence, by Theorem \ref{thm:main}, it follows that
\begin{align*}
&\|\nabla e^{-t\Delta}g\|_{L_p(\Omega_{n};\ell_2)}\leq c_{{\rm abs}}(p-1)^{-2}\|\Delta^{\frac1p}e^{-t\Delta}g\|_{L_p(\Omega_{n})}\\
\leq& c_{{\rm abs}}(p-1)^{-2}\|\Delta^{\frac1p}e^{-\Delta}\|_{L_p(\Omega_{n})\circlearrowleft}\|e^{-(t-1)\Delta}g\|_{L_p(\Omega_{n})}\\
\leq & c_{{\rm abs}}(p-1)^{-\frac52}\|\Delta^{\frac1p}e^{-\Delta}\|_{L_p(\Omega_{n})\circlearrowleft}\cdot e^{1-t}\|g\|_{L_p(\Omega_{n})}
\end{align*}
for all $t\geq 1$. This is exactly the second inequality.
\end{proof}

\begin{proof}[Proof of Theorem \ref{thm:heat-smoothing}] The estimate for $t\geq 1$ follows from the second assertion in Lemma \ref{exponential decay lemma} (which, obviously, holds also for not necessarily mean zero function $g$).

At the same time, for $0<t\leq 1$ we have
$$
\|\nabla e^{-t\Delta}g\|_{L_p(\Omega_{n};\ell_2)}\leq c_{{\rm abs}}(p-1)^{-2}\|\Delta^{\frac1p}e^{-t\Delta}g\|_{L_p(\Omega_{n})}\leq A_pt^{-\frac1p}\|g\|_{L_p(\Omega_{n})}.
$$

We now turn to show that the $\frac{1}{p}$ appears in \eqref{eq:heat-smoothing} is optimal. Suppose the power bound holds with exponent $\gamma<\frac1p.$ Choose $\alpha\in(\gamma,\frac1p)$ and using the formula
$$\nabla\Delta^{-\alpha}=\frac{1}{\Gamma(\alpha)}\int_0^{\infty}t^{\alpha-1}\nabla e^{-t\Delta}dt,$$
we obtain
\begin{align*}
&\|\nabla\Delta^{-\alpha}\|_{L_p(\Omega_{n};\ell_2)}\\
\leq& \frac{1}{\Gamma(\alpha)}\int_0^{\infty}t^{\alpha-1}\|\nabla e^{-t\Delta}g\|_{L_p(\Omega_{n};\ell_2)}dt\\
=&\frac{1}{\Gamma(\alpha)}\int_0^1t^{\alpha-1}\|\nabla e^{-t\Delta}g\|_{L_p(\Omega_{n};\ell_2)}dt+\frac{1}{\Gamma(\alpha)}\int_1^{\infty}t^{\alpha-1}\|\nabla e^{-t\Delta}g\|_{L_p(\Omega_{n};\ell_22)}dt\\
\leq &\frac{C_p}{\Gamma(\alpha)}\int_0^1t^{\alpha-\gamma-1}dt\cdot\|g\|_{L_p(\Omega_{n})}+\frac{B_p}{\Gamma(\alpha)}\int_1^{\infty}t^{\alpha-1}e^{-t}dt\cdot\|g\|_{L_p(\Omega_{n})}\\
\leq &\left(\frac{C_p}{\Gamma(\alpha)(\alpha-\gamma)}+B_p\right)\cdot\|g\|_{L_p(\Omega_{n})}.
\end{align*}
Thus, we established boundedness of the operator $\nabla\Delta^{-\alpha}$ on $L_p(\Omega_{n}).$ Using the optimality result established in Subsection \ref{optimality on exponent}, we conclude that $\alpha\geq\frac1p$. This contradicts the choice of $\alpha.$
\end{proof}

We next turn to polynomial inequalities. For each $n\in \mathbb{N}$ and $d\in [n]$, we define
\[
\mathcal{P}^{\leq d}(\Omega_{n})=\{f:\Omega_{n}\to \mathbb{C}:\widehat{f}(A)=0,~\mbox{if }|A|>d\},
\]
\[
\mathcal{P}^{<d}(\Omega_{n})=\{f:\Omega_{n}\to \mathbb{C}:\widehat{f}(A)=0,~\mbox{if }|A|\geq d\},
\]
and
\[
\mathcal{T}^{\geq d}(\Omega_{n})=\{f:\Omega_{n}\to \mathbb{C}: \widehat{f}(A)=0,~\mbox{if }|A|< d\}.
\]
On the hypercube, the Walsh degree plays the role of the algebraic degree in classical Bernstein--Markov theory.  The problem is closely connected, at the level of spectral calculus, with the heat-smoothing conjecture of Mendel and Naor\cite{M-N2014}.  Their conjecture concerns the opposite spectral regime: for functions whose Walsh spectrum is supported on levels at least $d$, one seeks dimension-free exponential decay of the form
\[
\|e^{-t\Delta}f\|_{L_{p}(\Omega_{n})}\le C_p e^{-c_pdt}\|f\|_{L_{p}(\Omega_{n})},
\]
for $f\in \mathcal{T}^{\geq d}(\Omega_{n})$. A weaker formulation at the generator level is the reverse Bernstein--Markov estimate for $\Delta$ on tail spaces. Eskenazis and Ivanisvili \cite{E-I2020} explored the heat smoothing conjecture and the Bernstein--Markov estimate in detail, drawing on duality, approximation theory, and complex analysis. The connection relevant here concerns how the heat semigroup interacts with a spectral cutoff. In particular, Eskenazis and Ivanisvili \cite[Theorem 14]{E-I2020} proved the following improved Bernstein--Markov type inequality for the gradient of functions.

\begin{prop}[Eskenazis-Ivanisvili]
For $p\in(1,\infty)$, let $\theta_{p}=2\arcsin\left(\frac{1\sqrt{p-1}}{p}\right)$. Then there exists a constant $C_{p}>0$ such that for every $n\in \mathbb{N}$ and $d\in [n]$, the following inequality holds
\begin{equation}\label{additional log term}
\|\nabla f\|_{L_{p}(\Omega_{n};\ell^{n}_{2})}\leq C_{p}d^{\frac{2}{p}-\frac{\theta_{p}}{p\pi}}\log(d+1)\|f\|_{L_{p}(\Omega_{n})},\qquad 1<p\leq 2,
\end{equation}
and
\[
\|\nabla f\|_{L_{p}(\Omega_{n};\ell^{n}_{2})}\leq C_{p}d^{1-\frac{\theta_{p}}{2\pi}}\|f\|_{L_{p}(\Omega_{n})}\qquad 2\leq p<\infty,
\]
where $f\in\mathcal{P}^{\leq d}(\Omega_{n})$.
\end{prop}

For $1<p<2$, the logarithm term in \eqref{additional log term} can be removed by the sharp fractional Riesz estimate (Theorem \ref{thm:main}).  Indeed, Eskenazis and Ivanisvili first proved the sharp estimate
\begin{equation}
\label{eq:EI-laplacian-BM}
  \|\Delta f\|_{L_p(\Omega_n)}
  \le 10d^{\alpha_p}\|f\|_{L_p(\Omega_n)},
  \qquad
  \alpha_p=2-\frac{\theta_p}{\pi},
\end{equation}
and then used a supercritical estimate
\begin{equation}
\label{eq:old-supercritical}
  \|\nabla h\|_{L_{p}(\Omega_{n})}
  \lesssim_p \frac1\varepsilon
  \|\Delta^{1/p+\varepsilon}h\|_{L_{p}(\Omega_{n})},
  \qquad \varepsilon>0.
\end{equation}
Combining \eqref{eq:EI-laplacian-BM} and \eqref{eq:EI-laplacian-BM} with interpolation to the order $1/p+\varepsilon$ yields a factor
$d^{\alpha_p(1/p+\varepsilon)}$. Choosing $\varepsilon\asymp1/\log(d+1)$ in this factor yields exactly the logarithmic loss in \eqref{additional log term}.  This observation explains why the endpoint $1/p$ in Theorem~1.1 is the natural input for removing the logarithm

The factor $\log(d+1)$ was subsequently removed by Volberg \cite{V2024} by a different complex analytic method. Hie proof starts from a probabilistic representation of the derivatives of the heat semigroup \cite{IVHV2020}. For $f\in \mathcal{P}^{\leq d}(\Omega_{n})$, writing
\[
 F_f(z,\varepsilon)=
 \sum_{S\subseteq[n]}z^{|S|}
 \widehat f(S)\varepsilon^S,
\]
he then applies complex hypercontractivity and a subharmonic
maximum-principle argument to the function
\[
 z\mapsto
 \log\|\nabla F_f(z,\cdot)\|_{L_p}.
\]
A suitably chosen lens domain and the asymptotic behavior of its Green
function at the corner $z=1$ yield directly
\[
 \|\nabla f\|_{L_{p}(\Omega_{n};\ell^{n}_{2})} \lesssim_p d^{\frac{2}{p}-\frac{\theta_p}{p\pi}}\|f\|_{L_{p}(\Omega_{n})},
\]
thereby avoiding the logarithmic loss. In particular, Volberg's argument bypasses the sharp fractional Riesz estimate altogether.

With the help of Theorem \ref{thm:main}, the same logarithm-free inequality follows by a much shorter proof. In our approach, we first pass from $\Delta$ to $\Delta^{1/p}$ using the scalar Bernstein--Markov estimate, and then pass from $\Delta^{1/p}$ to $\nabla$ using the sharp fractional Riesz estimate.

\begin{prop}[Volberg]\label{prop:log-free-BM}
For $1<p<2$, there exists a constant $C_{p}>0$ such that, for every
$n\in\mathbb N$, every $d\in  [n]$, the following inequality holds
\[
 \|\nabla f\|_{L_{p}(\Omega_n)}
 \leq
 C_{p}
 d^{\frac{2}{p}-\frac{\theta_{p}}{p\pi}}
 \|f\|_{L_{p}(\Omega_n)},
\]
for $f\in \mathcal{P}^{\leq d}(\Omega_{n})$, where $\theta_{p}=2\arcsin\left(\frac{1\sqrt{p-1}}{p}\right)$.
\end{prop}
\begin{proof}
Let $ \alpha_{p}=2-\frac{\theta_{p}}{\pi}$ and recall here that the following holds for $f\in \mathcal{P}^{\leq d}(\Omega_{n})$
\[
 \|\Delta f\|_{L_{p}(\Omega_{n})} \leq 10d^{\alpha_{p}}\|f\|_{L_{p}(\Omega_{n})}.
 \tag{1}
 \label{eq:BM-Delta}
\]
Note that
\begin{equation}\label{eq:fractional-moment-BM}
 \|\Delta^{\beta}g\|_{L_{p}(\Omega_{n})} \lesssim_{\beta} \|\Delta g\|_{L_{p}}^{\beta}\|g\|_{L_{p}(\Omega_{n})}^{1-\beta},
\end{equation}
for $\beta\in (0,1)$ and $g:\Omega_{n}\to \mathbb{C}$.  We substitute $\beta=\frac{1}{p}$ in \eqref{eq:fractional-moment-BM} and get
\begin{equation}\label{eq:fractional-Bernstein-BM}
\begin{split}
\|\Delta^{1/p}f\|_{L_{p}(\Omega_{n})}&\lesssim_{p}\|\Delta f\|_{L_{p}(\Omega_{n})}^{1/p}\|f\|_{L_{p}(\Omega_{n})}^{1-1/p}\\
&\leq C_{p} d^{\alpha_{p}/p}\|f\|_{L_{p}(\Omega_{n})}.
\end{split}
\end{equation}
It now remains to apply Theorem \ref{thm:main} to \eqref{eq:fractional-Bernstein-BM} to obtain the desired logarithm free inequality.
\end{proof}

The Bernstein--Markov estimates considered above concern functions whose
Walsh spectrum is bounded from above.  We now conclude this paper with the complementary high-frequency regime of tail spaces, which is the natural setting of the heat-smoothing problem of Mendel and Naor \cite{M-N2014}. Recall here that Ivanisvili and Nazarov \cite{I-N2022} obtained the following heat-smoothing estimate: for each $1<p<\infty$ and $ f\in\mathcal{T}^{\geq d}(\Omega_{n})$, 
\begin{equation}\label{eq:tail-heat-smoothing}
 \|e^{-t\Delta}f\|_{L_{p}(\Omega_{n})}\le C_q\exp\!\left[-c_qd\min\left\{t,t^{1/\vartheta_p}\right\}\right]\|f\|_{L_{p}(\Omega_{n})},
\end{equation}
where $\vartheta_p=1-\frac{2}{\pi}\arctan\!\left(\frac{|p-2|}{2\sqrt{p-1}}\right)$.

The sharp fractional Riesz estimate (i.e., Theorem \ref{thm:main}) allows us to convert this scalar spectral inequality into a first-order estimate for the operator $\nabla\Delta^{-1}$.

\begin{prop}
\label{prop:tail-gradient-potential}
Let $1<p<2$, let $p'=p/(p-1)$.  Then, for every $d\ge1$ and every
$h\in \mathcal{T}^{\geq d}(\Omega_{n})$,
\begin{equation}\label{eq:tail-gradient-potential}
 \|\nabla\Delta^{-1}h\|_{L_p(\Omega_n;\ell_2)}
 \lesssim_p
 d^{-\vartheta_p/p'}
 \|h\|_{L_p(\Omega_n)},
\end{equation}
where $\vartheta_q$ comes from \eqref{eq:tail-heat-smoothing}.
\end{prop}

\begin{proof}
By the sharp fractional Riesz estimate of Theorem \ref{thm:main}, we have
\begin{equation}
 \|\nabla u\|_{L_p(\Omega_n;\ell_2)}
 \lesssim_p
 \|\Delta^{1/p}u\|_{L_p(\Omega_n)},
 \qquad 1<p<2.
 \label{eq:critical-Riesz-tail-proof}
\end{equation}
Applying \eqref{eq:critical-Riesz-tail-proof} to
$u=\Delta^{-1}h$ gives
\begin{equation}\label{eq:gradient-potential-factor}
\begin{split}
\|\nabla\Delta^{-1}h\|_{L_p(\Omega_n;\ell_2)}
 &\leq \frac{c_{\mathrm{abs}}}{(p-1)^{2}}\|\Delta^{1/p}\Delta^{-1}h\|_{L_p(\Omega_n)}\\
 &=\frac{c_{\mathrm{abs}}}{(p-1)^{2}}\|\Delta^{-1/p'}h\|_{L_p(\Omega_n)},
\end{split}
\end{equation}
where we used $\frac{1}{p^{\prime}}=1-\frac{1}{p}$. Thus the problem is reduced exactly to a negative fractional-power estimate on the tail space.

We now put $\beta=\frac1{p^{\prime}}$. By Minkowski's inequality and
\eqref{eq:tail-heat-smoothing},
\begin{align}
 \|\Delta^{-\beta}h\|_{L_{p}(\Omega_{n})}
 &\le
 \frac{A_p}{\Gamma(\beta)}
 \|h\|_{L_{p}(\Omega_{n})}
 \int_0^\infty
 t^{\beta-1}
 \exp\!\left[
   -a_pd\min\{t,t^{1/\vartheta_p}\}
 \right]dt.
 \label{eq:negative-tail-integral}
\end{align}
We split the integral at $t=1$.  Since $\vartheta_p\le1$, for $0<t\le1$ we have $\min\{t,t^{1/\vartheta_p}\}=t^{1/\vartheta_p}$. Therefore, after the change of variables
$u=a_pd\,t^{1/\vartheta_p}$,
\begin{align}
 \int_0^1
 t^{\beta-1}e^{-a_pd t^{1/\vartheta_p}}\,dt
 &=
 \vartheta_p
 (a_pd)^{-\beta\vartheta_p}
 \int_0^{a_pd}
 u^{\beta\vartheta_p-1}e^{-u}\,du
 \notag\\
 &\le
 \vartheta_p
 \Gamma(\beta\vartheta_p)
 a_p^{-\beta\vartheta_p}
 d^{-\beta\vartheta_p}.
 \label{eq:small-time-tail-integral}
\end{align}
For $t\ge1$ the minimum equals $t$.  Since
$0<\beta<1$, we have $t^{\beta-1}\le1$ on $[1,\infty)$, and hence
\begin{align}
 \int_1^\infty
 t^{\beta-1}e^{-a_pd t}\,dt
 &\le
 \int_1^\infty e^{-a_pd t}\,dt
 =
 \frac{e^{-a_pd}}{a_pd}
 \lesssim_p
 d^{-\beta\vartheta_p}.
 \label{eq:large-time-tail-integral}
\end{align}
Combining
\eqref{eq:negative-tail-integral}--\eqref{eq:large-time-tail-integral}
gives
\begin{equation}
 \|\Delta^{-\beta}h\|_{L_{p}(\Omega_{n})}
 \lesssim_{p,\beta}
 d^{-\beta\vartheta_p}\|h\|_{L_{p}(\Omega_{n})}.
 \label{eq:negative-tail-beta}
\end{equation}
Recalling $\beta=1/p'$ and inserting \eqref{eq:negative-tail-beta} into
\eqref{eq:gradient-potential-factor} yields
\[
 \|\nabla\Delta^{-1}h\|_{L_p(\Omega_n;\ell_2)}
 \lesssim_p
 d^{-\vartheta_p/p'}\|h\|_{L_p(\Omega_n)},
\]
which proves \eqref{eq:tail-gradient-potential}.
\end{proof}

\appendix

\section{Complex interpolation between BMO and Hilbert spaces}\label{sec:product-interpolation}

We here provide the proof for Theorem \ref{thm:product-interpolation}. 
For a coefficient family $a=(a_{I,J})_{I,J\in\mathcal{D}_n} \subset \mathbb{C}$ define its discrete product tent square function by
$$S(a)^2=\sum_{I,J\in\mathcal{D}_n}|a_{I,J}|^2(h_I^2\otimes h_J^2).$$
For $1\leq r<\infty$, define
\begin{equation*}
\|a\|_{\T_u}=\|S(a)\|_{L_u(\Omega_n\times\Omega_n)}.
\end{equation*}
At the endpoint $r=\infty$, we define
$$\|a\|_{\T_{\infty}}^2=\sup_{\emptyset\neq U\subset \Omega_n\times\Omega_n}
 \frac1{(\mu_n\otimes\mu_n)(U)}\sum_{\substack{I,J\in\mathcal{D}_n\\ I\times J\subset U}}|a_{I,J}|^2.$$
We call $\T_u$ the \emph{tent sequence spaces}.
For $r=2$, the tent norm has the particularly simple form
\begin{equation*}
\|a\|_{\T_2}^{2}=\sum_{I,J\in\mathcal{D}_n}|a_{I,J}|^2.
\end{equation*}

We now identify the product tent sequence spaces exactly with the Haar
coefficient representation of the relevant function spaces. 
Given $F\in L_1(\Omega_n\times\Omega_n)$, define the coefficient map
\begin{equation*}
\mathcal{C}:F\to \left(
\langle F,h_I\otimes h_J\rangle
\right)_{I,J\in\mathcal{D}_n}.
\end{equation*}

\begin{lem}\label{lem:haar-tent}
The coefficient map is an isometric bijection between $L_2(\Omega_n\times\Omega_n)$ and $\T_2$ (respectively, between ${\rm BMO}(\Omega_n\times\Omega_n)$ and $\T_{\infty}$).
\end{lem}
\begin{proof} The first assertion follows from the fact that $\{h_I\otimes h_J\}_{I,J\in\mathcal{D}_n}$ is an orthonormal basis in $L_2(\Omega_n\times\Omega_n)$ (see Lemma \ref{orthonormal basis lemma}). The second assertion is obvious.
\end{proof}

Recall that, in Section \ref{sec:preliminaries}, the {\it increasing} filtration $\{\mathcal{F}_k\}_{k=0}^n$ are defined  by setting $\mathcal{F}_0=\{\Omega_{n},\emptyset\}$, and 
$$\mathcal{F}_k=\sigma(\mathbb{D}_k), \quad 1\leq k\leq n.$$
Let $E_{\mathcal{F}_k}$ be the conditional expectation with respect to $\mathcal{F}_{k}$.

The following lemma is standard.
\begin{lem}\label{conditional expectation of haar} If $I\in\mathbb{D}_k$ for $0\leq k\leq n-1,$ then $h_I\in\mathcal{F}_{k+1}$ and $E_{\mathcal{F}_k}h_I=0.$ 
\end{lem}

\begin{lem}\label{s7 pointwise equality} We have the pointwise identity
\begin{equation*}
S(\mathcal{C}F)=\left(\sum_{k,l=-1}^{n-1}\Big|\Big((E_{\mathcal{F}_{k+1}}-E_{\mathcal{F}_k})\otimes(E_{\mathcal{F}_{l+1}}-E_{\mathcal{F}_l})\Big)F\Big|^2\right)^{\frac12}.
\end{equation*}
Here, we use a shorthand $E_{\mathcal{F}_{-1}}=0.$
\end{lem}
\begin{proof} If $k,l\geq0,$ then it follows from Lemma \ref{conditional expectation of haar} that
$$\Big((E_{\mathcal{F}_{k+1}}-E_{\mathcal{F}_k})\otimes(E_{\mathcal{F}_{l+1}}-E_{\mathcal{F}_l})\Big)F=\sum_{I\in\mathbb{D}_k,J\in\mathbb{D}_l}\langle F,h_I\otimes h_J\rangle (h_I\otimes h_J).$$
Obviously,
\begin{align*}
&\Big|\sum_{I\in\mathbb{D}_k,J\in\mathbb{D}_l}\langle F,h_I\otimes h_J\rangle (h_I\otimes h_J)\Big|^2\\
=& \sum_{I_1,I_2\in\mathbb{D}_k,J_1,J_2\in\mathbb{D}_l}\overline{\langle F,h_{I_1}\otimes h_{J_1}\rangle}\cdot \langle F,h_{I_2}\otimes h_{J_2}\rangle (h_{I_1}h_{I_2}\otimes h_{J_1}h_{J_2}).
\end{align*}
Note that
$$h_{I_1}h_{I_2}=0,\quad I_1,I_2\in\mathbb{D}_k,\quad I_1\neq I_2.$$
It follows that
$$\Big|\sum_{I\in\mathbb{D}_k,J\in\mathbb{D}_l}\langle F,h_I\otimes h_J\rangle (h_I\otimes h_J)\Big|^2=\sum_{I\in\mathbb{D}_k,J\in\mathbb{D}_l}|\langle F,h_I\otimes h_J\rangle|^2\cdot (h_I^2\otimes h_J^2).$$
In other words,
$$\Big|\Big((E_{\mathcal{F}_{k+1}}-E_{\mathcal{F}_k})\otimes(E_{\mathcal{F}_{l+1}}-E_{\mathcal{F}_l})\Big)F\Big|^2=\sum_{I\in\mathbb{D}_k,J\in\mathbb{D}_l}|\langle F,h_I\otimes h_J\rangle|^2\cdot(h_I^2\otimes h_J^2).$$

Similarly,
$$\Big|\Big(E_{\mathcal{F}_0}\otimes(E_{\mathcal{F}_{l+1}}-E_{\mathcal{F}_l})\Big)F\Big|^2=\sum_{J\in\mathbb{D}_l}|\langle F,h_{\emptyset}\otimes h_J\rangle|^2\cdot (h_{\emptyset}^2\otimes h_J^2),$$
$$\Big|\Big((E_{\mathcal{F}_{k+1}}-E_{\mathcal{F}_k})\otimes E_{\mathcal{F}_0})\Big)F\Big|^2=\sum_{I\in\mathbb{D}_k}|\langle F,h_I\otimes h_{\emptyset}\rangle|^2\cdot (h_I^2\otimes h_{\emptyset}^2),$$
$$\Big|\Big(E_{\mathcal{F}_0}\otimes E_{\mathcal{F}_0})\Big)F\Big|^2=|\langle F,h_{\emptyset}\otimes h_{\emptyset}\rangle|^2.$$
Summing those equalities, we complete the proof.
\end{proof}

For $1\le u<\infty$, we use the notation
$$\|F\|_{H_u}=\left\|\left(\sum_{k,l=-1}^{n-1}\Big|\Big((E_{\mathcal{F}_{k+1}}-E_{\mathcal{F}_k})\otimes(E_{\mathcal{F}_{l+1}}-E_{\mathcal{F}_l})\Big)F\Big|^2\right)^{\frac12}\right\|_u.$$

\begin{lem}\label{hardy space vs tent space}
For every $1\leq u<\infty$, we have $\|\mathcal{C}F\|_{\T_u}=\|F\|_{H_u}.$
\end{lem}
\begin{proof} The assertion follows immediately from Lemma \ref{s7 pointwise equality}.
\end{proof}

For $f\in L_1(\Omega_n\times\Omega_n),$ define its maximal function by
$$(\mathcal{M}f)(x,y)=\sup_{\substack{I,J\in\mathcal{D}_n\\ h_I(x)\neq0\\ h_J(y)\neq 0}}\int_{\Omega_n\times\Omega_n}(|f|\cdot (h_I^2\otimes h_J^2)),\quad x,y\in\Omega_n.$$
Note that $h_I^2=\mu_n(I)^{-1}\1_I$ for each $I\in\mathcal{D}_n$. Then $\mathcal{M}$ is actually the Doob maximal function for two-parameter martingales.
The following strong $L_p$-boundedness of the maximal operator is taken from \cite[Proposition 3.4]{We1994}.

\begin{lem}[Weisz]\label{maximal function norm estimate}
For every $1<r<\infty$, we have
$$\|\mathcal{M}\|_{L_r(\Omega_n\times\Omega_n)\circlearrowleft}\leq \left(\frac{r}{r-1}\right)^2.$$
\end{lem}

In what follows, we use the notation: for each $ I,J\in\mathcal{D}_n$ and  $f\in L_1(\Omega_n\times\Omega_n)$, set
$$\fint_{I\times J}f=\int_{\Omega_n\times\Omega_n}(f\cdot (h_I^2\otimes h_J^2))=\frac{1}{\mu_n(I)\mu_n(J)}\int_{I\times J} f.$$

\begin{lem}\label{carleson lemma}
	Let $0<s<\infty$. For positive sequence $b\in\T_{\infty}$ and for positive function $f\in L_s(\Omega_n\times\Omega_n),$ we have
\begin{equation}\label{eq:layer-cake-carleson}
\sum_{I,J\in\mathcal{D}_n}b_{I,J}^2\Big(\fint_{I\times J}f\Big)^s\leq \|b\|_{\T_{\infty}}^2
 \|\mathcal{M}f\|_{L_s(\Omega_n\times\Omega_n)}^s.
\end{equation}
\end{lem}
\begin{proof} For $t>0,$ let
$$A_t= \{I\times J: I,J\in\mathcal{D}_n, \fint_{I\times J}f>t\},\quad U_t=\bigcup_{I\times J\in A_t} I\times J.$$
We claim that $U_t\subset\{\mathcal{M}f>t\}.$ Indeed, if  $(x,y)\in U_t,$ then there exist $I,J\in\mathcal{D}_n,$ such that $\fint_{I\times J}f>t$ and $(h_I\otimes h_J)(x,y)\neq 0.$ The second condition means $h_I(x)\neq0$ and $h_J(y)\neq 0.$ Hence,
$$(\mathcal{M}f)(x,y)\geq\int_{\Omega_n\times\Omega_n}(f\cdot (h_I^2\otimes h_J^2))=\fint_{I\times J}f>t.$$
This immediately yields the claim.
Therefore, we have
$$\sum_{\substack{I,J\in\mathcal{D}_n\\ I\times J\in A_t}}b_{I,J}^2=\sum_{\substack{I,J\in\mathcal{D}_n\\ I\times J\subset U_t}}b_{I,J}^2\leq (\mu_n\otimes\mu_n)(U_t)\|b\|_{\T_{\infty}}^2\leq (\mu_n\otimes\mu_n)(\{\mathcal{M}f>t\})\|b\|_{\T_{\infty}}^2.$$
Here, the first inequality follows from the definition of the set $U_t,$ while the second one follows from the definition of $\|\cdot\|_{\T_{\infty}}.$

Note that
$$\Big\|\Big(\fint_{I\times J}f\Big)_{I,J\in\mathcal{D}_n}\Big\|_{\ell_s}=\int_0^{\infty}st^{s-s}|A_t|dt.$$
Thus,
\begin{align*}
\sum_{I,J\in\mathcal{D}_n}b_{I,J}^2\Big(\fint_{I\times J}f\Big)^s&=\int_0^{\infty}st^{s-1}\Big(\sum_{\substack{I,J\in\mathcal{D}_n\\ I\times J\in A_t}}b_{I,J}^2\Big)dt\\
&\leq\int_0^{\infty}st^{s-1}\Big((\mu_n\otimes\mu_n)(\{\mathcal{M}f>t\})\cdot \|b\|_{\T_{\infty}}^2\Big)dt\\
&=\|b\|_{\T_{\infty}}^2\|\mathcal{M}f\|_{L_s(\Omega_n\times\Omega_n)}^s.
\end{align*}
\end{proof}

\begin{lem}\label{s7 holder lemma} For positive sequences $a\in\T_2$ and $b\in\T_{\infty},$ we have
$$\|a^{\frac{2}{q}}b^{1-\frac{2}{q}}\|_{\T_q}\leq c_{{\rm abs}}q\|a\|_{\T_2}^{\frac{2}{q}}\|b\|_{\T_{\infty}}^{1-\frac{2}{q}}.$$
\end{lem}
\begin{proof} Let $0\leq f\in L_{\frac{q}{q-2}}(\Omega_n\times\Omega_n)$ satisfy $\|f\|_{L_{\frac{q}{q-2}}(\Omega_n\times\Omega_n)}=1.$ 
We have
$$\int S^2(a^{\frac{2}{q}}b^{1-\frac{2}{q}})f=\sum_{I,J\in\mathcal{D}_n}a_{I,J}^{\frac{4}{q}}b_{I,J}^{2-\frac{4}{q}}\fint_{I\times J}f.$$
Using (discrete) H\"older inequality, we estimate
$$\int S^2(a^{\frac{2}{q}}b^{1-\frac{2}{q}})f\leq \Big(\sum_{I,J\in\mathcal{D}_n}a_{I,J}^2\Big)^{\frac{2}{q}}\Big(\sum_{I,J\in\mathcal{D}_n}b_{I,J}^2\Big(\fint_{I\times J}f\Big)^{\frac{q}{q-2}}\Big)^{1-\frac{2}{q}}.$$
Using Lemma \ref{carleson lemma} with $s=\frac{q}{q-2},$ we write
$$\int S^2(a^{\frac{2}{q}}b^{1-\frac{2}{q}})f\leq \|a\|_{\T_2}^{\frac{4}{q}}\|b\|_{\T_{\infty}}^{2-\frac4q}\|\mathcal{M}f\|_{L_{\frac{q}{q-2}}(\Omega_n\times\Omega_n)}.$$
Taking the supremum over $f$ running through the positive unit ball of $L_{\frac{q}{q-2}(\Omega_n\times\Omega_n)},$ we write
$$\|S^2(a^{\frac{2}{q}}b^{1-\frac{2}{q}})\|_{L_{\frac{q}{2}}(\Omega_n\times\Omega_n)}\leq  \|a\|_{\T_2}^{\frac{4}{q}}\|b\|_{\T_{\infty}}^{2-\frac4q}\|\mathcal{M}\|_{L_{\frac{q}{q-2}}(\Omega_n\times\Omega_n)\circlearrowleft}.$$
In other words,
$$\|S(a^{\frac{2}{q}}b^{1-\frac{2}{q}})\|_{L_q(\Omega_n\times\Omega_n)}\leq  \|a\|_{\T_2}^{\frac{2}{q}}\|b\|_{\T_{\infty}}^{1-\frac2q}\|\mathcal{M}\|_{L_{\frac{q}{q-2}}(\Omega_n\times\Omega_n)\circlearrowleft}^{\frac12}.$$
The assertion follows now from Lemma \ref{maximal function norm estimate}.
\end{proof}

\begin{thm}\label{thm:tent}
Let $2\leq q<\infty$ and $\theta=1-\frac{2}{q}.$ We have
$$\|\cdot\|_{\T_q}\leq c_{{\rm abs}}q\|\cdot\|_{[\T_2,\T_{\infty}]_{\theta}}.$$
\end{thm}
\begin{proof} Clearly, $\T_2=\ell_2$ possesses Radon-Nikodym property. By Theorem \ref{lem:lattice}, 
$$[\T_2,\T_{\infty}]_{\theta}=\T_2^{1-\theta}\T_{\infty}^{\theta}$$
isometrically. The assertion follows now from Lemma \ref{s7 holder lemma}.
\end{proof}

We now identify the Hardy norm with the ordinary $L_q$ norm, which follows from Burkholder-Gundy inequality for double dyadic martingales of Pipher \cite[Euqation (2.4)]{Pip1986}.

\begin{lem}[Pipher]\label{lem:product-square}
For every $2\leq q<\infty,$ we have
$$\|F\|_{L_q(\Omega_n\times\Omega_n)}\leq c_{{\rm abs}}q\|F\|_{H_q(\Omega_n\times\Omega_n)}.$$
The constant in the right hand side is optimal.
\end{lem}

\begin{proof}[Proof of Theorem \ref{thm:product-interpolation}] Using Lemmas \ref{lem:haar-tent} and the basic properties of complex interpolation, we write
$$\|F\|_{[L_2(\Omega_n\times\Omega_n),\BMO(\Omega_n\times\Omega_n)]_{\theta}}=\|\mathcal{C}F\|_{[\T_2,\T_{\infty}]_{\theta}}.$$
On the other hand, Lemma \ref{hardy space vs tent space} asserts that
$$\|F\|_{H_q(\Omega_n\times\Omega_n)}=\|\mathcal{C}F\|_{\T_q}.$$
The assertion follows now by combining the above equalities with Theorem  \ref{thm:tent} and Lemma \ref{lem:product-square}.
\end{proof}

\section*{Acknowledgment}

This work was supported by the National Natural Science Foundation of China
(Grant Nos. 12125109 \& W2411005); the Natural Science Foundation of Hunan
Province (Grant Nos: 2025ZYJ002, 2024JJ1010 \& 2024RC3040); the Scientific Research Fund of Hunan Provincial Education Department (Grant Nos. 25A0009
\& 25B0008). 

During the preparation of this work, the authors used GPT-5.6 Sol to improve the exposition of the manuscript and to verify calculations. After using this tool, the authors reviewed and edited the resulting content as necessary and take full responsibility for the content of the manuscript.

\end{document}